\documentclass{elsarticle}

\usepackage{tikz}
\usetikzlibrary{cd}

\usepackage{amsmath,amssymb,amsfonts,mathrsfs,mathtools, braket}
\usepackage{enumerate}
\newtheorem{theorem}{Theorem}[section]
\newtheorem{lemma}[theorem]{Lemma}
\newtheorem{corollary}[theorem]{Corollary}
\newtheorem{proposition}[theorem]{Proposition}

\newproof{proof}{Proof}
\newdefinition{definition}[theorem]{Definition}
\newdefinition{example}[theorem]{Example}
\newdefinition{remark}[theorem]{Remark}
\newdefinition{assumption}[theorem]{Assumption}

\numberwithin{equation}{section}

\allowdisplaybreaks

\begin{document}

\title{Phase retrieval for affine groups over prime fields}

\author{ David Bartusel, Hartmut F\"{u}hr, Vignon Oussa}

\address{D. Bartusel, H.~F{\"u}hr\\
Lehrstuhl f\"ur Geometrie und Analysis \\
RWTH Aachen University \\
 D-52056 Aachen\\
 Germany \\ email: david.bartusel@mathga.rwth-aachen.de, fuehr@mathga.rwth-aachen.de}
 
\address{V. Oussa\\ Department of Mathematics \\ Bridgewater State University \\ MA 02234 \\ USA \\ email: vignon.oussa@bridgew.edu}
\begin{abstract}
We study phase retrieval for group frames arising from permutation representations, focusing on the action of the affine group of a finite field. We investigate various versions of the phase retrieval problem, including conjugate phase retrieval, sign retrieval, and matrix recovery. Our main result establishes that the canonical irreducible representation of the affine group $\mathbb{Z}_p \rtimes \mathbb{Z}_p^\ast$ (with $p$ prime), acting on the vectors in $\mathbb{C}^{p}$ with zero-sum, has the strongest retrieval property, allowing to reconstruct matrices from scalar products with a group orbit consisting of rank-one projections. We explicitly characterize the generating vectors that ensure this property, provide a linear matrix recovery algorithm and explicit examples of vectors that allow matrix recovery. We also comment on more general permutation representations.   
\end{abstract}
\maketitle
\global\long\def\with{\,\middle|\,}

\hyphenation{an-iso-tropic}
\noindent \textbf{\small Keywords:}{\small {}  phase retrieval; conjugate phase retrieval; sign retrieval; matrix recovery; Pauli pairs; group frame; group Fourier transform}{\small \par}

\noindent \textbf{\small AMS Subject Classification:}{\small {} 42C15;
42A38; 94A12; 20B20}{\small \par}

\section{Introduction}
The phase retrieval problem for frames was introduced in \cite{MR2224902}, with motivation coming from applications such as speech recognition and imaging. Since its introduction to the field of mathematical signal processing, phase retrieval has developed into an autonomous area of mathematical research with rich connections to diverse mathematical fields such as complex analysis, optimization, numerical analysis, algebraic geometry, harmonic analysis and representation theory, as witnessed by the review papers \cite{MR4094471,MR4189292} and their extensive lists of references. 

The phase retrieval problem can be briefly described as follows: Given a system $\Psi = (\psi_i)_{i \in I} \subset \mathbb{C}^d$, we aim to reconstruct $f \in \mathbb{C}^d$, up to a scalar factor $\alpha \in \mathbb{T}$, from the family $\left( |\langle f, \psi_i \rangle| \right)_{i \in I}$. Here $\mathbb{T}$ denotes the multiplicative group of complex numbers with modulus one. Clearly, a necessary condition for the system $\Psi$ is that ${\rm span}(\Psi) = \mathbb{C}^d$, i.e. the linear coefficient map $f \mapsto (\langle f, \psi_i \rangle )_{i \in I}) \in \mathbb{C}^I$ is necessarily injective. This observation motivates the name ``phase retrieval'': Solving the phase retrieval problem for a system $\Psi$ spanning $\mathbb{C}^d$ amounts to reconstructing the phases of the expansion coefficients $(\langle f, \psi_i \rangle)_{i \in I}$ from their moduli (which is a nonlinear problem), and then inverting the coefficient map (which is a standard, linear procedure).

It is therefore clear that phase retrieval requires some redundancy of $\Psi$: A basis $\Psi$ does not admit phase retrieval, since each expansion coefficient with respect to a basis (in particular its phase) can be chosen independently of the remaining ones. The necessary degree of redundancy has been quantified fairly sharply: Generic systems $\Psi$ consisting of $\ge 4d-4$ vectors allow phase retrieval, and for certain dimensions $d$, phase retrieval for systems with less than $4d-4$ vectors is impossible \cite{MR3303679}.

A problem closely related to phase retrieval is \textit{matrix recovery} \cite{MR3032952,MR2549940}, asking whether any matrix $A \in \mathbb{C}^{d \times d}$ can be recovered from the family
\[
 \mathbb{C}^{d \times d} \ni A \mapsto (\langle A \psi_i, \psi_i \rangle)_{i \in I} \in \mathbb{C}^I
\] of scalar products. This injectivity property was discussed in \cite{MR2549940}  by the name of \textit{maximal span property}, picking up on earlier investigations of the problem under the name of 
\textit{quantum state tomography} in physics.

Any system $\Psi$ doing matrix recovery does phase retrieval. To see this, one only needs to realize that 
\[
 |\langle f, \psi \rangle|^2 = \langle (f \otimes f)(\psi), \psi \rangle~,
\] where $f \otimes f \in \mathbb{C}^{d \times d}$ is given by $(f \otimes f)(m,n) = f(m) \overline{f(n)}$. 
Hence if $\Psi$ does matrix recovery, $f \otimes f$ can be recovered from $(|\langle f, \psi_i \rangle|^2)_{i \in I}$, and $f$ can be recovered up to a phase factor from $f  \otimes f$. An interesting byproduct of this observation is that stability of linear inversion is well-understood and controllable, unlike the non-linear problem of phase recovery. However, there is also a price to pay in terms of redundancy: For dimension reasons, matrix recovery requires at least $d^2$ measurements.

Phase retrieval and matrix recovery are often studied under additional assumptions on the analyzed signal or matrix, such as sparsity or low rank; see e.g. \cite{MR3175089,MR3574562}. Here one does not aim for recovery guarantees for all possible signals, thereby allowing a reduction in the number of measurements. The constructions in this context, as well as the recovery guarantees, are typically of a probabilistic nature.

\subsection*{Overview of the paper}

In this paper we are interested in the deterministic construction of structured systems guaranteeing phase retrieval and matrix recovery, more precisely in the construction of group frames with phase retrieval. Group frames are characterized by the fact that they are generated as orbits $ \Psi = \pi(G) \psi$ of a (usually unitary) representation $\pi$ of a group $G$ acting on a suitably chosen vector $\psi$; see \cite{MR2964010} for group frames over finite groups, and \cite{MR2130226} for an extensive discussion of group frames over locally compact groups.

The chief example of group frames allowing phase retrieval is provided by the case where $G = \mathbb{H}$ is a Heisenberg group, and $\pi$ is a Schr\"odinger representation. This setting provided the earliest results concerning the use of group actions for quantum state tomography \cite{MR115648}. In the setting of Schr\"odinger representations, the phase retrieval properties of the system $\pi(G) \psi$ were seen to be closely related to properties of the \textit{ambiguity function} $A_\psi : \mathbb{H} \ni x \mapsto  \langle \psi, \pi(x) \psi \rangle$ (see \cite{MR4094471}). It was then  observed in \cite{MR3500231} (or rather, rediscovered, in the light of \cite{MR115648}) that for finite Heisenberg groups, proper assumptions on $A_\psi$ even allow to solve the \textit{matrix recovery problem}. A further investigation of this property for more general group frames was then provided in
\cite{MR3901918,MR3907853}, with emphasis on irreducible projective representations of finite abelian groups. A quite far-reaching extension of these results to irreducible representations of general nilpotent groups was recently obtained in \cite{FuehrOussaNilpotent}. These developments notwithstanding, the problem of phase retrieval for group frames is far from being fully understood. 

In this paper we focus on permutation representations, and in particular on actions of affine groups $\mathbb{Z}_p \rtimes \mathbb{Z}_p^*$ of finite fields $\mathbb{Z}_p$ (with $p$ prime). Our main results, Theorem \ref{thm:crit_op_ret}, Corollary \ref{cor:crit_op_ret} and Theorem \ref{thm:ex_or}, establish criteria for generating vectors guaranteeing matrix recovery, provide an explicit recovery formula, and simple concrete examples of generating vectors fulfilling these criteria. We also present a general recipe how to tackle the problem of matrix recovery for group frames (see Remark \ref{rem:Group_fourier}), and explain how the known instances (finite Heisenberg groups and affine groups) fit in this general framework (Remark \ref{rem:mr_Heisenberg}). Remarks \ref{rem:Pauli_pairs} and \ref{rem:pt_fourier} exhibit some connections of our results to intriguing questions in finite Fourier analysis. 

To put these results in perspective, we first point out that the retrieval problem for finite affine groups can be understood as a finite analog of the wavelet phase retrieval problem previously studied by Mallat and Waldspurger \cite{MR3421917}. At this point in time it is safe to say that explicit constructions of systems for matrix recovery, with linear reconstruction formulae, are only known for very specific systems of vectors, such as equiangular tight frames, or unions of $d+1$ mutually unbiased bases in $d$-dimensional Hilbert spaces, see \cite{MR1003014,MR2549940}. It is considered highly doubtful that every dimension $d$ accommodates $d+1$ mutually unbiased bases. Known explicit constructions of such systems exist only in very restricted dimensions, e.g. for $d$ equal to a prime power. By contrast to this, our representation theoretic construction of a system doing matrix recovery via the affine group of the field $\mathbb{Z}_p$ ($p$ prime) works in a space of dimension $d=p-1$, yielding a system of $d(d+1)$ vectors. Hence, while our argument does not rely on the notion of mutually unbiased bases, it is intriguing to note that the cardinalities of the resulting group frames obey the same dimension-dependent relationship. 

Throughout the paper, we also discuss variants of phase retrieval, such as sign and conjugate phase retrieval, and extensions to more general permutation representations. We present several results and examples of independent interest, mostly for the purpose of highlighting the challenges that the representation-theoretic understanding of phase retrieval properties faces. 

\section{Phase retrieval for group frames}

Throughout the paper, we let $\mathbb{K} \in \{ \mathbb{R},\mathbb{C} \}$. Given finite sets $X,Y$, we let $\mathbb{K}^{ X \times Y}$ denote the space of matrices $A = (A(x,y))_{x \in X, y \in Y}$, with entries in $\mathbb{K}$. We identify matrices with the induced linear maps $\mathbb{K}^Y \to \mathbb{K}^X$ if needed. We endow the matrix space $\mathbb{K}^{X \times Y}$ with the usual scalar product
\[
 \langle A, B \rangle = {\rm trace}(AB^*)~.
\]
We will use the convention $\mathbb{K}^d = \mathbb{K}^{\{0,\ldots,d-1 \}}$ for $d \in \mathbb{N}$. Special vectors in $\mathbb{K}^d$ are given by $\mathbf{1} = (1,\ldots,1)^T$ and $\delta_k \in \mathbb{K}^d$, $\delta_k(m) = 1$ for $k=m$ and $0$ elsewhere.

We now formally introduce a variety of retrieval properties. Sign and phase retrieval for frames were introduced by Balan et.al. \cite{MR2224902}. The intermediate notion of conjugate phase retrieval is due to Evans and Lai \cite{MR4029744}. Matrix recovery and associated linear reconstruction methods were introduced to this discussion in \cite{MR2549940}; see also \cite{MR3032952}.

\begin{definition} 
Let $\mathcal{H}$ denote a finite-dimensional real or complex Hilbert space. 
Given a system $\Psi = (\psi_i)_{i \in I} \subset \mathcal{H}$, we define the associated coefficient transform as
\[
 V_\Psi : \mathcal{H} \to \mathbb{C}^I ~,~ f \mapsto (\langle f, \psi_i \rangle)_{i \in I}~.
\] Recall that $V_\Psi$ is injective iff $(\psi_i)_{i \in I}$ is total.  
Given a complex Hilbert space $\mathcal{H}$, we say that the system $\Psi \subset \mathcal{H}$ 
\textbf{does phase retrieval} if the map 
\[
 \mathcal{A}_\Psi :\mathcal{H}/ \mathbb{T} \ni f \mapsto |V_\Psi f| \in \mathbb{R}^I
\] is one-to-one. Here $\mathcal{H}/\mathbb{T}$ denotes the orbit space under the canonical action of $\mathbb{T}$ on $\mathcal{H}$, and $|V_\Psi f|$ is obtained by pointwise application of the modulus function to $V_\Psi f$.

Given a real Hilbert space $\mathcal{H}$, and $\Psi = (\psi_i)_{i \in I} \subset \mathcal{H}$, we define the \textbf{sign retrieval property} by requiring that 
\[ 
\mathcal{A}_\Psi : \mathcal{H} / \{ \pm 1\} \to   \mathbb{R}^I~,~ f \mapsto |V_\Psi f| \]
is one-to-one. 

We  define the intermediate property of \textbf{conjugate phase retrieval} as follows: We assume that $\mathcal{H} \subset \mathbb{C}^X$ is a complex subspace, where $X$ is finite, and $\mathbb{C}^X$ is endowed with the usual scalar product. Furthermore, we assume $\Psi \subset \mathbb{R}^X \cap \mathcal{H}$. Then $\Psi$ \textbf{does conjugate phase retrieval} if the equation $\mathcal{A}_\Psi(f) = \mathcal{A}_\Psi(g)$ is equivalent to the existence of $\alpha \in \mathbb{T}$ such that either $f = \alpha g$ or $f = \alpha \overline{g}$ holds, where complex conjugation is understood to be applied entrywise.

Finally, we say that the system $\Psi \subset \mathcal{H}$ \textbf{does matrix recovery} if the map
\[
 \mathcal{L}(\mathcal{H}) \ni A \mapsto V_{\Psi \otimes \Psi} A \in \mathbb{C}^I,~V_{\Psi \otimes \Psi} A(i) = \langle A \psi_i, \psi_i \rangle
\] is one-to-one. Here $\mathcal{L}(\mathcal{H})$ denotes the space of linear mappings $\mathcal{H} \to \mathcal{H}$. 
\end{definition}
%

\begin{remark} \label{rem:gen_retrieval}
 Consider the following statements, for a system $\Psi = (\psi_i)_{i \in I} \subset \mathbb{C}^X$:
 \begin{enumerate}
  \item[1.)] $\Psi$ does matrix recovery.
  \item[2.)] $\Psi$ does phase retrieval.
  \item[3.)] $\Psi$ does conjugate phase retrieval.
  \item[4.)] $\Psi$ does sign retrieval.
  \item[5.)] $\Psi$ spans $\mathbb{C}^X$.
 \end{enumerate}
Then one has the implications $1.) \Rightarrow 2.) \Rightarrow 5.)$, as well as $3.) \Rightarrow 4.) \Rightarrow 5.)$. Observe that conditions 3.), 4.) refer to real-valued systems, which never do phase retrieval,  due to the equation $\overline{\langle f, \psi \rangle} = \langle \overline{f}, \psi \rangle$ holding for real-valued $\psi$.  

%
%

\end{remark}

\begin{definition} \label{defn:pr_rep}
 Let $\pi$ denote a representation of a finite group $G$ on the finite-dimensional Hilbert space $\mathcal{H}_\pi$ over $\mathbb{K}$. Given $\psi \in \mathcal{H}_\pi$, the associated group frame is given by $\Psi = \pi(G) \psi = (\pi(x) \psi)_{x \in G}$. We call $\psi$ the \textit{generating vector} of the system $\Psi$. 
 
 Given $\psi$, we write 
 \[
  V_\psi : \mathcal{H}_\pi \to \mathbb{C}^G~,~ V_\psi f(x) = \langle f, \pi(x) \psi \rangle~. 
 \] In case we want to emphasize the dependence on the representation $\pi$, we will also use $V_\psi^\pi = V_\psi$. 

  We say that $\pi$ \textbf{does phase/conjugate phase/sign retrieval} if there exists $\psi \in \mathcal{H}$ such that $\Psi = (\pi(x) \psi)_{x \in G}$ does phase/conjugate phase/sign retrieval. 
 
In the case of conjugate phase retrieval, we additionally assume that $\mathcal{H}_\pi = \mathbb{C}^X$ for suitable $X$, and that $\pi(G) (\mathbb{R}^X) \subset \mathbb{R}^X$ holds. 
 
 Similarly, $\pi$ \textbf{does matrix recovery} if there exists $\psi \in \mathcal{H}$ such that $\Psi = (\pi(x) \psi)_{x \in G}$ does matrix recovery. 
\end{definition}

The following observation reformulates matrix recovery for group frames. Recall that a vector $\psi \in \mathcal{H}_\pi$ is called \textit{cyclic} for $\pi$ if $V_\psi^\pi$ is one-to-one. With this definition in mind, the proposition is seen to be a mere reformulation.  
\begin{proposition} \label{prop:mr_conjrep}
Let $\pi$ denote a representation of a group $G$ acting on $\mathbb{C}^X$. 

Let by $\varrho = (\pi \otimes \overline{\pi})|_{\Delta_G}$, i.e. the representation of $G$ acting on $\mathbb{C}^{X \times X}$ by
\[
 \varrho(x) A = \pi(x) \cdot A \cdot \pi(x)^*~.
\]
Then $\psi \in \mathcal{H}$ does matrix recovery iff $\psi \otimes \psi$ is a cyclic vector for $\varrho$. 
\end{proposition}

We next prove a basic feature of the sets of vectors $\psi$ that guarantee one of the retrieval properties.
In the following lemma, a set $U \subset \mathbb{K}^d$ is called \textit{real Zariski open} (by slight abuse of terminology) if there exist finitely many functions $f_1,\ldots,f_k$ on $\mathbb{K}^d$ that are polynomials in the real and imaginary parts of the coordinates such that 
\[
 U = \{ x \in \mathbb{K}^d : \exists \text{ }    1 \le i \le k~:~ f_i(x) \not= 0 \}~.
\] The terminology naturally extends to subset of $\mathbb{K}^{d\times d}$.

\begin{lemma} \label{lem:Zariski}
 Assume that $\pi$ does sign/conjugate phase/phase/matrix recovery. Then the set of vectors $\psi \in \mathcal{H}_\pi$ that guarantee this property is real Zariski open.
\end{lemma}
\begin{proof}
 We identify a system $\Psi = (\psi_i)_{i=1}^m$ of vectors in $\mathbb{K}^d$ with the matrix $\Psi \in \mathbb{K}^{d \times m}$. Then the set of all $\Psi$ that perform phase or sign retrieval is real Zariski open in $\mathbb{K}^{d \times m}$, by \cite{MR3303679}. The analogous observation has been made for conjugate phase retrieval in \cite{MR4029744}. 
 
 In order to apply these observations to group frames, it is enough to note that the map $\psi \mapsto V_\psi$ is conjugate-linear, and that the preimage of a real Zariski open set under such a map is real Zariski open.
 
 In the case of matrix recovery, the map $\psi \mapsto V_{\psi \otimes \psi}^\varrho$ is polynomial, and by definition, $\psi$ does matrix recovery iff the linear map $V^\varrho_{\psi \otimes \psi}$ has rank $d_\pi^2$. However, the set of maximal rank matrices is real Zariski open, being defined by the nonvanishing of at least one $d_\pi^2 \times d_\pi^2$-subdeterminant. Hence the set of vectors $\psi$ guaranteeing matrix recovery is the polynomial preimage of a real Zariski open set, hence real Zariski open. 
\end{proof}
Note that real Zariski open sets are open in the standard topology, and either empty or of full Lebesgue measure. This shows that once a retrieval property is established for $\pi$, almost every vector $\psi$ guarantees it. However, verifying either that the real Zariski open set is nonempty, or that a given concrete vector is contained in it, remains a considerable challenge, and the results in the present papers bear witness to this fact.

\begin{remark} \label{rem:gf_retrieval}
So far, phase retrieval for group frames $\Psi = \pi(G) \psi$ has been studied almost exclusively for Schr\"odinger representations of (finite or infinite) Heisenberg groups. An extension of these results to projective representations of abelian groups was performed in \cite{MR3901918}. Phase retrieval for the continuous wavelet transform in dimension one was studied in \cite{MR3421917,alaifari2020phase}. Here the underlying representation is the quasiregular representation of the affine group $\mathbb{R} \rtimes \mathbb{R}^*$, and it is fair to say that the phase retrieval problem even for this basic representation is far from being completely understood. Our paper provides an analysis for the finite analogs of $\mathbb{R} \rtimes \mathbb{R}^*$. 

The implications from Remark \ref{rem:gen_retrieval} immediately yield the implications $1.) \Rightarrow 2.) \Rightarrow 5.)$, and $3.) \Rightarrow 4.) \Rightarrow 5.)$ between the following statements. 
\begin{enumerate}
  \item[1.)] $\pi$ does matrix recovery.
  \item[2.)] $\pi$ does phase retrieval.
  \item[3.)] $\pi$ does conjugate phase retrieval.
  \item[4.)] $\pi$ does sign retrieval.
  \item[5.)] $\pi$ is cyclic.
 \end{enumerate}
Unlike in the general case considered in Remark \ref{rem:gen_retrieval}, the implication $2.) \Rightarrow 3.)$ does make sense as soon as $\mathcal{H}_\pi = \mathbb{C}^X$, and $\pi(G) (\mathbb{R}^X) \subset \mathbb{R}^X$ holds. Note that here $2.)$ asks for the existence of a particular complex-valued generating vector, whereas $3.)$ requires a suitable real-valued generating vector. The relationship between conditions $2.)$ and $3.)$ is currently open. 
 
At this point it is also worth noting that not all properties are invariant under unitary equivalence. While it is easily verified that matrix and phase retrieval of a representation are preserved under unitary equivalence, the property that $\pi(G) (\mathbb{R}^X) \subset \mathbb{R}^X$ depends on the realization of the representation. This follows from the fact that most intertwining operators do not commute with taking pointwise conjugates or real parts. A concrete example will be given by the representation $\widehat{\pi}_0$ studied more closely in Section \ref{sect:matrix_retrieval}, which is equivalent to the permutation representation $\pi_0$.  The permutation representation satisfies $\pi_0(G) (\mathbb{R}^X) \subset \mathbb{R}^X$, but its unitary equivalent $\widehat{\pi}_0$ does not. 
 
 We finally point out that not all cyclic representations do phase retrieval. As an extreme case in point, we can take the regular representation $\lambda_G$ of a finite group $G$, which is cyclic. Any cyclic vector $\psi$ for $\lambda_G$ necessarily fulfills $V_\Psi (\mathbb{C}^G) = \mathbb{C}^G$ (for dimension reasons), and the map $F \mapsto |F|$ is clearly not injective on $\mathbb{C}^G$.
\end{remark}

\begin{remark}
 \label{rem:proj_rep}
 Recent research activity has extended the original setup to include \textit{projective representations} of finite groups $G$, see e.g. \cite{MR3901918}. Recall that by definition, a (unitary) projective representation of a finite group $G$ is a map $\pi: G \to \mathcal{U}(\mathcal{H}_\pi)$ such that 
 \[
\forall x,y \in G:   \pi(xy) = \alpha(x,y) \pi(x) \pi(y)~,\alpha: G \times G \to \mathbb{T}~.
 \]
 This extension can be useful for the treatment of special cases, such as the Schr\"odinger representation of a finite Heisenberg group $\mathbb{H}$, which can be alternatively viewed as a projective representation of the quotient group $\mathbb{H} / Z(\mathbb{H})$, where $Z(\mathbb{H})$ denotes the center of $\mathbb{H}$. 
 
 However, it is well-known that every projective representation $\pi$ of a finite group $G$ can be lifted to a standard unitary representation $\pi'$ of a finite central extension $G'$ of $G$, and the system $\pi'(G') \psi$ differs from $\pi(G) \psi$ only by the inclusion of a fixed number of scalar multiples of $w$, for every vector $w \in \pi(G) \psi$. None of the retrieval properties discussed here is affected by this change. Hence the extension to projective representations ultimately does not enlarge the realm of available constructions in any truly substantial manner. 
\end{remark}

\section{Setup: Affine groups and permutation representations}

Let $p>2$ denote a prime, and let $\mathbb{Z}_p = \mathbb{Z}/p \mathbb{Z}$, the associated finite field of order $p$. We typically denote $\mathbb{Z}_p = \{ 0,1, \ldots,p-1 \}$, and understand arithmetic operations on this set to be defined modulo $p$. $\mathbb{Z}_p^* = \mathbb{Z} \setminus \{ 0 \}$ then denotes the multiplicative group of $\mathbb{Z}_p$. The affine group of $\mathbb{Z}_p$ is given by the semidirect product $G = \mathbb{Z}_p \rtimes \mathbb{Z}_p^\ast$, using the canonical multiplicative action of $\mathbb{Z}_p^\ast$ on $\mathbb{Z}_p$. I.e. as a set, $G = \mathbb{Z}_p \times \mathbb{Z}_p^\ast$, with group law $(k,l)(k',l') = (k+lk',ll')$. Unless otherwise specified, the symbol $G$ will always refer to this semidirect product. 

Given $n \in \mathbb{N}$, we let $S(n)$ denote the symmetric group of $\{ 0,\ldots, n-1 \}$. Note that since
$G = \mathbb{Z}_p \rtimes \mathbb{Z}_p^\ast$ acts faithfully on $\mathbb{Z}_p = \{0, 1,\ldots,p-1 \}$ via $(k,l).m = k+lm$, we may regard $G$ as a subgroup of $S(p)$. In particular, one has $ S(3) = \mathbb{Z}_3 \rtimes \mathbb{Z}_3^\ast $. 

The \textbf{quasiregular representation}  $\Pi$ of $S(n)$ acts on $\mathbb{K}^n$ by
\begin{equation}
 (\Pi (h) f)(m) = f(h^{-1} m)~.
\end{equation}
In the case of $n=p$ a prime number, the restriction of $\Pi$ to the subgroup $G = \mathbb{Z}_p \rtimes \mathbb{Z}_p^*$ will be denoted by $\pi = \Pi|_G$, and it is given explicitly by 
\begin{equation} \label{eqn:qr}
 (\pi(k,l) f)(m) = f(l^{-1}(m-k))~.
\end{equation}
Note that our notation does not explicitly distinguish the real and complex cases. In fact, we have $\Pi(S(n)) (\mathbb{R}^n) \subset \mathbb{R}^n$, a fact which allows to discuss conjugate phase and sign retrieval for $\Pi$ and its restrictions. 

\begin{remark} \label{rem:finite_fields}
 The semi-direct product $\mathbb{Z}_p \rtimes \mathbb{Z}_p^\ast$ and its quasi-regular representation were studied in \cite{MR1310638} as a tool for the construction of wavelets on finite fields. In this sense, this paper can also be seen as dealing with a finite field analog of wavelet phase retrieval, see e.g.  \cite{MR3421917,alaifari2020phase}.
\end{remark}

We let $\mathbf{1} = (1,\ldots,1)^T \in \mathbb{K}^n$. 
We make the following observation:
\begin{lemma}
The subspaces $\mathcal{H}_1 = \mathbb{K} \cdot \mathbf{1}$ and $\mathcal{H}_0 = \mathcal{H}_1^\bot$ are invariant under $\Pi$, and therefore also under $\pi$. The associated restrictions $\pi_0,\pi_1$ of $\pi$ acting on $\mathcal{H}_0,\mathcal{H}_1$ respectively are irreducible.
\end{lemma}

In this paper we wish to address the various retrieval properties for $\pi_0$ and/or $\pi$. Ideally, we would like to have explicit and sharp criteria for generating vectors guaranteeing any of the properties, as well as concrete examples. We first note a simple observation regarding $\pi$. 

\begin{corollary}
 The representation $\pi$ does not do matrix recovery.
\end{corollary}
\begin{proof}
Matrix recovery is equivalent to stating that ${\rm span} \{ \pi(x) \psi \otimes \pi(x) \psi : x \in G \} = \mathbb{C}^{p \times p}$. For dimension reasons, this requires $p^2 \le |G| = p(p-1)$, which is obviously false. 
\end{proof}

Note that this argument does not apply to the codimension one subspace $\mathcal{H}_0$, and we will indeed show that $\pi_0$ does matrix recovery. 

\section{Sign retrieval for $\pi_0$}

In this section we consider the sign retrieval properties of the permutation representation $\pi_0$ acting on the real vector space $\mathcal{H}_0 = \{ x \in \mathbb{R}^p~:~ \langle x, \mathbf{1} \rangle = 0 \}$. 

\begin{definition} \text{ }
\begin{enumerate}
\item Let $\mathcal{H}$ denote a finite-dimensional Hilbert space, and
$(\psi_i)_{i \in I} \subset \mathcal{H}$. Then $(\psi_i)_{i \in I}$ has the \textit{complement property} iff for all $S \subset I$, either ${\rm span}(\psi_i)_{i \in S} = \mathcal{H}$, or ${\rm span}(\psi_i)_{i \in S \setminus I} =  \mathcal{H}$. 
 
 \item The family has \textit{full spark} if ${\rm dim}(\mathcal{H}) = d \in \mathbb{N}$, and for every subset $S \subset I$ of cardinality $d$, the system $(\psi_i)_{i \in S}$ spans $\mathcal{H}$.
 \end{enumerate}
\end{definition}

The following result are contained in \cite{MR3202304}; see Theorems 3 and 7. 
\begin{theorem} \label{thm:char_pr_cp}
 Let $\Psi = (\psi_i)_{i \in I} \subset \mathcal{H}$ denote a system of vectors contained in the finite-dimensional Hilbert space $\mathcal{H}$. 
 \begin{enumerate}
  \item[(a)] If $\Psi$ does phase or sign retrieval, then $\Psi$ has the complement property.
  \item[(b)] If $\mathcal{H}$ is a real Hilbert space, then $\Psi$ does sign retrieval iff it has the complement property. 
 \end{enumerate}
\end{theorem}

The following theorem is implied by Theorem 10 in \cite{MR4121508}.
\begin{theorem}
The quasiregular representation $\pi_0$ on $\mathcal{H}_0$ has full spark, which means that there exists a real-valued vector $\psi \in \mathcal{H}_0$ such that $(\pi(x) \psi)_{x \in G}$ has full spark. Any such vector does sign retrieval.  
\end{theorem}

\begin{proof}
The first statement is Theorem 10 in \cite{MR4121508}; a closer inspection of its proof shows that a vector whose orbit is a full spark frame can indeed be chosen to be real-valued. Finally, since $G$ has $p(p-1) > 2p-1$ elements, every full spark frame indexed by $G$ necessarily has the complement property. Hence, Theorem \ref{thm:char_pr_cp} yields sign retrieval.
\end{proof}

%
%

\section{Conjugate phase retrieval for doubly transitive permutation groups}

From now on we let $\mathcal{H}_0 = \langle \mathbf{1} \rangle \subset \mathbb{C}^p$. 
Throughout this section, we may assume that $\Gamma \subset S(n)$ is a subgroup of the symmetric group. Additionally, we also assume that $n\ge 3$ to avoid trivialities. Recall next, that $\Gamma$ is called $k$-fold transitive if for every pair of $k$-tuples $u,v \in \{0,\ldots, n-1 \}^k$, where both $u$ and $v$ have pairwise distinct entries, there exists $h \in \Gamma$ such that 
\[
 \forall i=1,\ldots,k~:~ v_i = u_{h(i)}~.
\] $2$-fold transitive groups are also called \textbf{doubly transitive}. An important class of doubly transitive groups is given by $G = \mathbb{Z}_p \rtimes \mathbb{Z}_p^\ast$, for $p \ge 3$. 

We intend to prove that $2$-fold transitive groups do conjugate phase retrieval. The argument will be based on a particular choice of generating vector, and the following, well-known property of the complex plane.

\begin{remark} \label{rem:euclid_geom}
 Let $(y_k)_{k=0,\ldots,n-1}$, $(z_k)_{k=0,\ldots,n-1} \in \mathbb{C}^n$ satisfy 
 \begin{equation}
 \label{eqn:part_isom}
  \forall l,k = 0,\ldots,n-1~:~|y_l-y_k| = |z_l -z_k|~.
 \end{equation}
 Then there exists an isometry $\varphi : \mathbb{C} \to \mathbb{C}$ with $\varphi(y_l) = z_l$, for all $l=0,\ldots,n-1$. 
 
Furthermore, the group of isometries of the Euclidean space $\mathbb{C} \cong \mathbb{R}^2$ coincides with the Euclidean motion group, i.e. it is the composition of a translation and an orthogonal matrix. In addition, any orthogonal mapping on $\mathbb{R}^2$ is given by (complex) multiplication with an element of $\mathbb{T}$, possibly followed by conjugation. 

In short, we have deduced from (\ref{eqn:part_isom}) that either 
 \[
  \forall l=0,\ldots, n-1~:~z_l = \alpha y_l + w
 \]
 or 
 \[
  \forall l=0,\ldots, n-1~:~ z_l = \alpha \overline{y_l} + w~,
 \]
holds, with $\alpha \in \mathbb{T}$ and $w \in \mathbb{C}$.
\end{remark}

We then have the following observation, which implies conjugate phase retrieval for $\pi_0$. 
\begin{theorem} \label{thm:dble_cpr}
 Let $\Gamma < S(n)$ denote a doubly transitive subgroup. Pick $k_0,l_0 \in \{0,\ldots, n-1 \}$ with $k_0 \not= l_0$, and let $\psi = \delta_{k_0} - \delta_{l_0}$. Then $\Pi(\Gamma) \psi$ does conjugate phase retrieval on the invariant subspace $\mathcal{H}_0$. 
\end{theorem}
\begin{proof}
 Let $(k,l)$ denote any distinct pair of elements, and let $h \in \Gamma$ such that $h(k_0) = k,h(l_0) = l$. Then $  V_\psi f (h) = f(k)-f(l)$ and this shows for all $f,g \in \mathcal{H}_0$:
\[
|V_\psi f| = |V_\psi g| \Longleftrightarrow \forall (k,l) \in \{0,\ldots, n-1 \}^2: |f(k)-f(l)| = |g(k)-g(l)|~.
\]

An application of Remark \ref{rem:euclid_geom} yields either $g(l) = \alpha f(l) + w$ or $g(l) = \alpha \overline{g(l)}+w$ holds for all $l$, with $\alpha \in \mathbb{T}$ and $w \in \mathbb{C}$. 

To complete the proof, it remains to show $w=0$, but this follows (for the case without conjugation) from $f,g \in \mathcal{H}_0$ via 
\[
 0 = \sum_{l=0}^{n-1} g(l) = \sum_{l=0}^{n-1} \alpha f(l) + w = nw + \alpha \sum_{l=0}^{n-1} f(l) = nw~. 
\] The second case, involving complex conjugation, is treated entirely analogously. This establishes conjugate phase retrieval.
\end{proof}

\begin{remark}
 We now turn to the question of conjugate phase retrieval for $\Pi$ on $\mathbb{C}^n$, using the vector 
\[
 \psi_1 = \underbrace{\delta_{k_0}-\delta_{l_0}}_{\psi_0} + n^{-1/2}\mathbf{1}~.
\]
Note that the projection of $\psi_1$ onto the subspace $\mathcal{H}_0$ does conjugate phase retrieval for the restriction of $\pi$ to that subspace, and the same can be said (for trivial resasons) of the orthogonal complement of $\mathcal{H}_0$, the one-dimensional space $\mathbb{C} \cdot \mathbf{1}$. Yet these properties are generally not enough to guarantee conjugate phase retrieval for the full space $\mathbb{C}^n$, as the following example for $n=3$ shows: 

Let $\xi = e^{2 \pi i/3}$, and define $y = (y_0,y_1,y_2) \in \mathbb{C}^3$ by letting
\[
  y_l = \xi^l~.
\] We then have $\sum_k y_k = 0$, and one easily verifies,
for all $x \in S(3)$,
\[
 |\langle y,\Pi(x) \psi_1 \rangle| = |1-\xi| = |\langle  \frac{|1-\xi|}{\sqrt{3}} \mathbf{1}, \Pi(x) \psi_1 \rangle|~. 
\] On the other hand, the vectors $w = \frac{|1-\xi|}{\sqrt{3}} \mathbf{1}$ and $y= (y_0,y_1,y_2)$ are clearly distinct, even orthogonal.

Note that this example does not yet establish that $\Pi$ does not do conjugate phase retrieval on $\mathbb{C}^{3}$, as it only applies to a particular choice of generating vector $\psi_1$. But it does illustrate that the problem of conjugate phase retrieval on all of $\mathbb{C}^3$ cannot be addressed by being able to solve it on each invariant subspace $\mathcal{H}_0$,$\mathcal{H}_1$ separately, somewhat contrary to common (linear) representation theoretic wisdom. This phenomenon emphasizes the nonlinear nature of the conjugate phase retrieval problem. 

We do not know whether analogs of the counterexample constructed here are available in dimensions $n>3$. 
\end{remark}
As a further application of the basic proof principle exploited in this section, we have the following observation: 
\begin{proposition}
 Let $n \ge 3$, and let $\Gamma < S(n)$ denote a doubly transitive subgroup. Pick $k_0,l_0 \in \{0,\ldots, n-1 \}$ with $k_0 \not= l_0$, and let $\psi_1 = \delta_{k_0}-\delta_{l_0} + n^{-1/2} \mathbf{1}$. Then $\Pi(\Gamma) \psi_1$ does sign retrieval on $\mathbb{R}^n$. 
\end{proposition}
\begin{proof}
 Assume $f = f_0 + a \cdot n^{-1/2} \mathbf{1}$, $g = g_0 + b \cdot n^{-1/2} \mathbf{1}$ with $f_0,g_0 \in \mathbb{R}^n \cap \mathcal{H}_0$ and $a,b \in \mathbb{R}$, such that 
 \[
\forall x \in \Gamma~:~  |\langle f, \Pi(x) \psi_1 \rangle| = |\langle g, \Pi(x) \psi_1 | ~.
 \] Since $\Gamma$ is doubly transitive, this is equivalent to
 \begin{equation}
 \label{eqn:sign_ret}
  \forall l \not= k~:~|f(k)-f(l)|^2+a^2 = |g(k)-g(l)|^2+b^2~,~a(f(k)-f(l)) = b(g(k)-g(l))~.
 \end{equation}
 
 Assuming that $f$ is constant then yields either $b=0$, or that $g$ is constant. The latter implies $|a|=|b|$, and $g = \pm f$, as desired (since $f_0 = g_0=0$). If $b=0$, we get
 \[
  |g(k)-g(l)|=|a|~,~l \not= k~.
 \] However, there do not exist $n \ge 3$ equidistant points on the real line, hence this case can be excluded.
 
 Hence it remains to address the case that $f$ is not constant. Assuming $a=0$ then requires that $b(g(k)-g(l)) = 0$. If $b=0$, we obtain
 \[
\forall l \not= k~:~  |f(k)-f(l)| = |g(k)-g(l)|
 \] and thus $g = \pm f$ by Theorem \ref{thm:dble_cpr}. If $b\not=0$, it follows that $g$ is constant, and thus 
 \[
  \forall l \not=k ~:~ |f(k)-f(l)|^2 = b^2-a^2
 \] and since $f$ is not constant, $b^2-a^2>0$ follows. As above, the fact that there are no $n \ge 3$ equidistant points in $\mathbb{R}$ excludes this possibility.
 
 The final case to consider is therefore that $f$ is not constant and $a \not= 0$. This implies that 
 \[
  \forall l \not= k ~:~ g(k)-g(l) = \frac{b}{a} (f(k)-f(l))~.
 \]
Plugging this into the first equation of (\ref{eqn:sign_ret}) and simplifying yields 
\[
 \forall l \not= k ~:~ |f(k)-f(l)|^2 = \frac{b^2-a^2}{1-b^2/a^2} = -a^2~, 
\] which contradicts our assumptions on $a$ and $f$. 
\end{proof}

\section{Matrix recovery for $\pi_0$}

\label{sect:matrix_retrieval}

In this section we establish the matrix recovery property for $\pi_0$. Recall that we are interested in finding specific cyclic vectors for the representation $\varrho_0$ obtained by conjugating with $\pi_0$. The basic strategy of the following is to first decompose $\varrho_0$ into suitable subrepresentations, with the help of an explicitly derived intertwining operator; see Proposition \ref{prop:decompose}. This decomposition  will provide fairly explicit access to criteria for cyclic vectors. The remaining task is to apply the intertwining operator from Proposition \ref{prop:decompose} to matrices of the type $A = \varphi \otimes \varphi$, and determine vectors $\varphi$ that fulfill the criteria for cyclic vectors. As we will see, the quadratic dependence of $A$ on $\varphi$ makes this final step a further  nontrivial obstacle. 

The starting point in this approach is to switch to a unitarily equivalent representation $\widehat{\pi}_0$, obtained by conjugating $\pi_0$ with the Fourier transform on $\mathbb{C}^p$. In the following, we will denote elements of $\mathbb{C}^p = \mathbb{C}^{\{ 0,\ldots,p-1 \}}$ as $f = (f(0),\ldots,f(p-1))$. With this notation, the Fourier transform is given by  
\begin{equation}
\label{eqn:F_additive}
 \mathcal{F} : \mathbb{C}^p \to \mathbb{C}^p~,~\mathcal{F}(f)(m) = p^{-1/2}\sum_{n=0}^{p-1} f(n) e^{-2 \pi i n m/p}~. 
\end{equation} We will also use the notation $\widehat{f} = \mathcal{F}(f)$. With the normalization used here, the Plancherel theorem becomes $\| \widehat{f} \|^2 = \| f \|^2$, whereas the convolution theorem for the convolution product
\[
 (f \ast g) (m) = \sum_{n=0}^{p-1} f(n) g(m-n)~. 
\]
reads as
\[
 (f \ast g)^\wedge(k) = p^{1/2} \widehat{f}(k) \widehat{g}(k)~. 
\]

The chief advantage of the Fourier transform is that it diagonalizes the translation action of the affine group. The effect of conjugation with $\mathcal{F}$ is spelled out in the following lemma.
\begin{lemma} \label{lem:pi_four}
 Define $\widehat{\pi} : G \to \mathcal{U}(\mathbb{C}^p)$, $\widehat{\pi}(k,l) = \mathcal{F} \pi(k,l) \mathcal{F}^*. $
 \begin{enumerate}
  \item[(a)] $\widehat{\pi}(k,l) f(m) = e^{-2 \pi i km/p} f(lm)$, with $\mathcal{F}(\mathcal{H}_1) = \mathbb{C} \cdot \delta_0$, $\mathcal{F}(\mathcal{H}_0) = \{ 0 \} \times \mathbb{C}^{\{ 1,\ldots,p-1 \}} \cong \mathbb{C}^{\{ 1,\ldots,p-1 \}} $. 
  \item[(b)] Let $\varrho_1$ denote the conjugation action of $\widehat{\pi}$ on $\mathbb{C}^{\{1,\ldots,p-1 \} \times \{ 1,\ldots, p-1 \}}$. Then one has, for all $A \in \mathbb{C}^{\{1,\ldots,p-1 \} \times \{ 1,\ldots, p-1 \}}$
  \[
   \left( \varrho_1(k,l) A \right) (m,n) = e^{-2 \pi i k (m-n)/p} A(lm,ln)~. 
  \]
 \end{enumerate}
\end{lemma}

\begin{proof}
 We include the computations proving the lemma for completeness. For part (a), we have for all $f \in \mathbb{C}^p$ that 
 \begin{eqnarray*}
  \mathcal{F} (\pi(k,l) f)(m) &  = & \sum_{n=0}^{p-1} f(l^{-1}(n-k)) e^{-2 \pi i n m/p} \\
  & = & \sum_{n=0}^{p-1} f(n) e^{-2 \pi i (ln+k) m/p} \\
  & = & e^{-2 \pi i km/p} \mathcal{F}(f)(lm) ~.
 \end{eqnarray*}
 The statements concerning $\mathcal{H}_0,\mathcal{H}_1$ are obvious. 

 For the proof of (b), we fix $A \in \mathbb{C}^{\{1,\ldots, p-1\} \times \{ 1,\ldots, p-1 \}}$ as well as $f \in \mathbb{C}^p$, and get by (a)
 \begin{eqnarray*}
 (A \cdot \widehat{\pi}(k,l)^* \cdot f)(r) & = & ((A \cdot \pi(-l^{-1}k,l^{-1}) \cdot f)(r)  \\
 & = & \sum_{n=1}^{p-1} A(r,n) f(l^{-1}n) e^{-2 \pi i (-l^{-1}k)n/p} \\
 & = & \sum_{n=1}^{p-1} A(r,ln) f(n) e^{2 \pi i kn/p} ~.
 \end{eqnarray*}
It follows that 
\begin{eqnarray*}
 (\pi(k,l) \cdot A \cdot \pi(k,l)^* \cdot u)  (m)  & = & 
 e^{-2\pi i km/p} \sum_{n=0}^{p-1} A(lm,ln) e^{2 \pi i kn/p} u(n) \\
 & = & \sum_{n=0}^{p-1} e^{-2 \pi i k(m-n)} A(lm,ln) u(n)~,
\end{eqnarray*}
which establishes (b). 
\end{proof}

\begin{remark}
Our explicit matrix recovery algorithm in Corollary \ref{cor:crit_op_ret} makes use of the group Fourier transform of $G$. In the following, we will briefly sketch the fundamentals of this transform, for a \textit{general} finite group $G$, and a concrete choice of \textit{unitary dual} $\widehat{G} = \left( (\sigma,\mathcal{H}_\sigma) \right)_{\sigma \in \widehat{G}}$ of $G$. By definition, the unitary dual is a system of representatives of the unitary equivalence classes of irreducible representations of $G$. Such a system exists for every finite group $G$, and it is necessarily finite.  For ease of notation, we assume $\mathcal{H}_\sigma = \mathbb{C}^{d_\sigma}$, i.e., $\sigma(x) \in \mathbb{C}^{d_\sigma \times d_\sigma}$. 

Given $F \in G^\mathbb{C}$, one can now define its \textit{group Fourier transform} as the matrix-valued tuple 
\[
 \mathcal{F}_G(F) = (\sigma(F))_{\sigma \in \widehat{G}} \in \prod_{\sigma \in \widehat{G}} \mathbb{C}^{d_\sigma \times d_\sigma}
\]
given by 
\begin{equation} \label{eqn:group_fourier}
 \mathcal{F}_G(F) (\sigma) = \sigma(F) = \sum_{g \in G} F(g) \sigma(g)~.  
\end{equation}
Observe that the Fourier transform (\ref{eqn:F_additive}) for the cyclic case is a special case of this construction, since the characters of the additive group $\mathbb{Z}_p$ are precisely given by 
\[
 \chi_k(l) = e^{-2 \pi i lk/p}~.
\] Note however that $\mathcal{F}$ in equation (\ref{eqn:F_additive}) is normalized by the factor $p^{-1/2}$, in order to obtain a unitary map, whereas we refrained from introducing a normalization to the group Fourier transform. 
In order to avoid confusion in the following, we will mostly stick to the notation $\sigma(F)$ instead of $\mathcal{F}_G(F)(\sigma)$, when we refer to the group Fourier transform.

The Fourier transform fulfills the \textit{Plancherel formula} 
\[
\| F \|_2^2 = |G|^{-1} \sum_{\sigma \in \widehat{G}} d_\sigma \| \sigma(F) \|^2
\] where the norm on the right hand side is the Frobenius norm. 
Its inverse is given explicitly by 
\begin{equation} \label{eqn:Fourier_inv}
 \mathcal{F}_G^{-1} (A_\sigma) (g) = |G|^{-1} \sum_{\sigma \in \widehat{G}} d_\sigma {\rm trace}(A_\sigma \sigma(g)^*)~,
\end{equation} for all $(A_\sigma)_{\sigma \in \widehat{G}} \in \prod_{\sigma \in \widehat{G}} \mathbb{C}^{d_\sigma \times d_\sigma}$.
For all these general facts, we refer the reader to \cite[Chapter 15]{MR1695775}. 
\end{remark}

Returning to the special case of the affine group $G = \mathbb{Z}_p \rtimes \mathbb{Z}_p^\ast$, we have the following description of its unitary dual.  
We refer to \cite[Chapter 17]{MR1695775} for a proof. 
\begin{lemma}
 The unitary dual of $G= \mathbb{Z}_p \rtimes \mathbb{Z}_p^\ast$ is given by the following set of representatives: The representation $\widehat{\pi}_0$, as well as the set $\{ \tilde{\chi} : \chi \in \widehat{\mathbb{Z}_p^*} \}$ of characters of the multiplicative group, lifted to the full group $G$ by letting
 \[
\tilde{\chi} (k,l) := \chi(l)~,\chi \in \widehat{\mathbb{Z}_p^*}~.
 \]
\end{lemma}

We thus obtain the following group Fourier transform for the affine group $G$: Given $F \in \mathbb{C}^G$, its group Fourier transform is the tuple 
\[
 \mathcal{F}(F) = \left( \tilde{\chi}(F) \right)_{\chi \in \widehat{\mathbb{Z}_p^\ast}} \cup (\widehat{\pi}_0(F)). 
\]
consisting of the scalars
\[
 \tilde{\chi}(F)  = \sum_{\ell \in \mathbb{Z}_p^\ast} \left( \sum_{k \in \mathbb{Z}_p} F(k,l) \right) \chi(l)~,
\] and the matrix
\[
 \widehat{\pi}_0(F) = \sum_{k,l} F(k,l) \widehat{\pi}_0(k,l) \in \mathbb{C}^{\{1,\ldots,p-1 \} \times \{1, \ldots, p-1 \}} 
\]

The next result provides a decomposition of $\varrho_1$ into suitable subrepresentations, which is an important step towards establishing matrix recovery. We remind the reader that all algebraic operations involving numbers $l,k,m,n \in \{0,\ldots,p-1 \}$ are understood modulo $p$. 

\begin{proposition} \label{prop:decompose}
 Define the unitary operator 
 \[
  S :  \mathbb{C}^{\{1,\ldots,p-1 \} \times \{ 1,\ldots, p-1 \}} \to  \mathbb{C}^{\{1,\ldots,p-1 \} \times \{ 1,\ldots, p-1 \}}
 \] by 
 \[
  (SA)(m,n) = \left\{ \begin{array}{ll} A(-m,-m) & \mbox{ for } n=1 \\ A(m (1-n)^{-1}, m n (1-n)^{-1}) & \mbox{ for } n \in \{ 2 ,\ldots, p-1 \} \end{array} \right.  
 \]
 Define $\varrho_2 : G \to \mathcal{U}(  \mathbb{C}^{\{1,\ldots,p-1 \} \times \{ 1,\ldots, p-1 \}} )$ by
 $\varrho_2(k,l) = S \circ \varrho_1(k,l) \circ S^*$. Then 
 \begin{equation} \label{eqn:varrho2}
  (\varrho_2(k,l) A)(m,n) = \left\{ \begin{array}{ll}  A(lm,n) & \mbox{ for } n=1 \\
  e^{-2 \pi i km/p} A(lm,n) &  \mbox{ for } n \in \{ 2 ,\ldots, p-1 \} \end{array} \right. 
 \end{equation}
Define $\mathcal{K}_1 \subset \mathbb{C}^{\{1,\ldots,p-1 \} \times \{ 1,\ldots, p-1 \}}$ as the subspace of matrices $A$ with $A(m,n) = 0$ for all $n \not= 1$, and let $\mathcal{K}_2 = \mathcal{K}_1^\bot \cong 
\mathbb{C}^{\{1,\ldots,p-1 \} \times \{ 2,\ldots, p-1 \}}$. These spaces are invariant under $\varrho_2$, and we have 
\[
 \varrho_2|_{\mathcal{K}_1} \simeq \tilde{\lambda}_{\mathbb{Z}_p^\ast} ~,~ \varrho_2|_{\mathcal{K}_2} \simeq \widehat{\pi}_0 \otimes \mathbf{1}_{\mathbb{C}^{\{ 2,\ldots,p-1 \}}}~. 
\]
where $\tilde{\lambda}_{\mathbb{Z}_p^\ast}$ is the regular representation of  $\mathbb{Z}_p^\ast$, lifted to $G$. 
\end{proposition}

\begin{proof}
It is clearly enough to prove (\ref{eqn:varrho2}), and we do this by explicit calculation. First, one quickly verifies that the map 
\[
 (m,n) \mapsto \left\{ \begin{array}{ll} (-m,-m) & \mbox{ for } n=1 \\ (m (1-n)^{-1}, m n (1-n)^{-1}) & \mbox{ for } n \in \{ 2 ,\ldots, p-1 \} \end{array} \right.
\] is a bijection on $\{1,\ldots,p-1 \} \times \{ 1,\ldots, p-1 \}$, with inverse given by
\[
 (m,n) \mapsto \left\{  \begin{array}{ll} (-m,1) & \mbox{ for } n=m \\ (m-n, m^{-1} n) & \mbox{ for } m \not= n  \end{array} \right.
\] Hence $S$ is indeed unitary, with 
\[
 (S^*A)(m,n) =  \left\{  \begin{array}{ll} A (-m,1) & \mbox{ for } n=m \\ A (m-n, m^{-1} n) & \mbox{ for } m \not= n  \end{array} \right.
\]
Fixing $A \in \mathbb{C}^{\{1,\ldots,p-1 \} \times \{ 1,\ldots, p-1 \}}$, we let $B_1 = \varrho_1(k,l) S^*(A)$ and 
$B_2 = S B_1 = \varrho_2(k,l) A$. We then get 
\[
 B_1(m,n) = e^{-2 \pi i k(m-n)/p} \left\{  \begin{array}{ll} A (l m,1) & \mbox{ for } n=m \\ A (\ell(m-n), m^{-1} n) & \mbox{ for } m \not= n  \end{array} \right.
\] Accordingly, for $m=1,\ldots,p-1$, 
\[
 B_2(m,1) = B_1(m,m) = A(lm,1)~.
\] For $n \in \{2,\ldots,p-1 \}$ we fix $(m_0,n_0) = (m (1-n)^{-1}, mn(1-n)^{-1})$, and get
\[
 m_0-n_0 = m(1-n)^{-1} - mn(1-n)^{-1} = m (1-n) (1-n)^{-1} = m~,~ m_0^{-1}n_0 = n
\]
which allows to compute
\[
 B_2(m,n) = B_1(m_0,n_0) =  e^{-2 \pi i k(m_0-n_0)/p} A(l (m_0-n_0), m_0^{-1} n_0) = e^{- 2 \pi i km/p}A(lm,n)~,
\] and (\ref{eqn:varrho2}) is established. 
\end{proof}

We next state criteria for matrix recovery.
\begin{theorem} \label{thm:crit_op_ret}
Let $\varphi \in \mathbb{C}^{\{1,\ldots,p-1\}}$. Then $\widehat{\pi}_0(G) \varphi$ does matrix recovery iff $\varphi$ fulfills the following conditions:
\begin{enumerate}
 \item[(i)] All entries of $c_{\varphi} \in \widehat{\mathbb{Z}_p^*}^\mathbb{C}$, defined 
 by
 \[
  c_{\varphi}(\chi) = \sum_{m =1}^{p-1} |\varphi(-l)|^2 \chi(l)
 \] are nonzero. 
 \item[(ii)] The matrix $B_\varphi \in \mathbb{C}^{\{1,\ldots,p-1 \} \times \{ 1,\ldots, p-2 \}}$ defined by
 \[
  B_\varphi(m,n) = \varphi( m n ) \overline{\varphi( m (n+1))}~
 \] has a left inverse $B_{\varphi}^\dagger \in  \mathbb{C}^{\{1,\ldots,p-2 \} \times \{ 1,\ldots, p-1 \}}$.
\end{enumerate}
\end{theorem}

The first and most involved step of the proof (given below) consists in explicitly inverting the linear operator $V_{\varphi \otimes \varphi}^{\varrho_1} : \mathbb{C}^{\{1,\ldots,p-1\} \times \{ 1, \ldots, p-1 \}} \to \mathbb{C}^G$, assuming that $\varphi$ fulfills conditions (i) and (ii). We give a detailed description of this procedure in the following corollary to the proof: 
\begin{corollary} \label{cor:crit_op_ret}
 Assume that $\varphi \in  \mathbb{C}^{\{1,\ldots,p-1\}}$ fulfills conditions (i) and (ii) from the previous theorem. 
An explicit algorithm for the inversion of the linear map $V_{\varphi \otimes \varphi}^{\varrho_1} : C^{\{1,\ldots,p-1\} \times \{ 1, \ldots, p-1 \}} \to \mathbb{C}^G$ is described as follows: Let $\Omega_0 \in \mathbb{C}^{\{1,\ldots,p-1 \} \times \{ 1,\ldots,p-1 \}}$ denote the matrix describing the linear map 
\[
 (\Omega_0 f)(m) = f(-m)~,
\]
and let $\Omega_1 \in \mathbb{C}^{\{1,\ldots, p-2\} \times \{2,\ldots, p-1 \}}$ be the matrix describing the linear map  
\[
 (\Omega_1 f)(n) = f(1+n^{-1})~.
\]

Let $A \in \mathbb{C}^{\{1,\ldots,p-1\} \times \{ 1,\ldots, p-1 \}}$ and let $F \in \mathbb{C}^G$ be defined as
\[
 F(k,l) = \left(  V_{\varphi \otimes \varphi}^{\varrho_1}A \right) (k,l)  = \langle A \widehat{\pi}_0(k,l) \varphi, \widehat{\pi}_0(k,l) \varphi \rangle~.
\]
Then $A$ can be recovered from $F$ by the following steps:
\begin{enumerate}
 \item Define $a_1 \in \mathbb{C}^{ \{ 1,\ldots,p-1 \}}$ by
 \[
  a_1(k) = \frac{1}{p(p-1)} \sum_{\chi \in \widehat{\mathbb{Z}_p^\ast}} c_{\varphi}(\chi)^{-1} \tilde{\chi}(F) \chi(k)~.
 \]
 \item Define $A_2' = p \cdot \widehat{\pi_0}(F) \cdot \Omega_0^T \cdot (B_{\varphi}^\dagger)^\ast \cdot \Omega_1 \in \mathbb{C}^{\{1,\ldots,p-1\} \times \{ 2,\ldots,p-1\}}$.
 \item With the unitary operator $S: \mathbb{C}^{\{1,\ldots,p-1 \} \times \{ 1,\ldots, p-1 \}} \to  \mathbb{C}^{\{1,\ldots,p-1 \} \times \{ 1,\ldots, p-1 \}}$ from Proposition \ref{prop:decompose}, we have 
 \[
  A = S^* \left( (a_1|A_2') \right)
 \]
 \end{enumerate}
 In particular, $f \in \mathbb{C}^{\{1,\ldots,p-1 \}}$ can be recovered up to a phase factor by computing $A$ from  $F = |V_\varphi^{\widehat{\pi_0}} f|^2$ using steps (1)-(3), and then determining $f$ as the eigenvector associated to the nonzero eigenvalue of $A$, normalized by $\|f\|_2^2 = {\rm trace}(A)$. 
\end{corollary}

\begin{proof}
Throughout the proof, the symbol $\mathbf{0}$ will denote zero matrices of varying sizes. Since these sizes can be derived from the context, we will refrain from using more specific notation.
%
%

The intertwining property of $S$ yields 
 \begin{eqnarray*}
  F(k,l) & = & \langle A, \pi(k,l) \varphi \otimes \pi(k,l) \varphi \rangle \\
  & = & \langle SA, \varrho_2(k,l) S(\varphi \otimes \varphi) \rangle~.
 \end{eqnarray*}

We write $SA = A_1 + A_2$ with $A_i \in \mathcal{K}_i$, and get
\[
 A_1 = \left( a_1| \mathbf{0} \right) ~,~ A_2 = \left( \mathbf{0} | A_2' \right)~,
\] with a vector $a_1 \in \mathbb{C}^{\{1,\ldots,p-1 \} }$, $A_2' = \mathbb{C}^{\{1,\ldots,p-1\} \times \{2,\ldots, p-1 \}}$. 

Similarly, we have 
\[ S(\varphi \otimes \varphi) = (c_\varphi|C_\varphi') = \underbrace{(c_\varphi|\mathbf{0})}_{\in \mathcal{K}_1} + \underbrace{(\mathbf{0}|C_\varphi')}_{\in \mathcal{K}_2}
\]
where $(c_\varphi|C'_\varphi) = C_\varphi$ is obtained by plugging
\[
 (\varphi \otimes \varphi)(m,n) = \varphi(m) \overline{\varphi(n)}
\]
into the definition of $S$, which leads to  
\begin{equation} \label{eqn:Cphi}
 C_\varphi(m,n) = \left\{ \begin{array}{ll} |\varphi(-m)|^2 & \mbox{ for } n=1 \\ \varphi(m (1-n)^{-1}) \overline{\varphi(m n (1-n)^{-1})} & \mbox{ for } n \in \{ 2 ,\ldots, p-1 \} \end{array} \right.  
\end{equation}
In particular, 
\[
 c_\varphi(n) = |\varphi(-n)|^2~.
\]  

The overall proof strategy for the inversion formula is to recover both $a_1$ and $A_2'$ from $F$ via group Fourier transform, which allows to reconstruct $SA$. It then only remains to apply $S^*$ to the result.

By definition of the various matrices involved, we have
\[
 F(k,l) = \underbrace{\langle A_1,\varrho_2(k,l) (c_\varphi|\mathbf{0}) \rangle}_{=F_1(k,l)} + 
 \underbrace{\langle A_2, \varrho_2(k,l) (\mathbf{0}|C_\varphi') \rangle}_{=F_2(k,l)}~.
\] 
Now equation (\ref{eqn:varrho2}) yields
\begin{equation} \label{eqn:F1}
 F_1(k,l) = \sum_{m=1}^{p-1} a(m) |\varphi(-lm)|^2~, 
\end{equation} whereas
\begin{equation} \label{eqn:F2}
 F_2(k,l) = \langle A_2, \widehat{\pi}_0(k,l) (\mathbf{0},C_\varphi') \rangle = {\rm trace}\left(A_2 \cdot (\mathbf{0}|C_\varphi')^* \cdot \widehat{\pi}_0(k,l)^* \right)~.
\end{equation}
This reveals $F_1$ as a convolution product over the factor group $\mathbb{Z}_p^\ast$, whereas $F_2$ is seen to be a special instance of the Fourier inversion formula (\ref{eqn:Fourier_inv}). In particular, we have
\begin{equation} \label{eqn:Fi_orthog}
 F_1 \in {\rm span}\{\tilde{\chi} : \chi \in \widehat{\mathbb{Z}_p^*} \}~,~ F_2 ~\bot ~ {\rm span}\{\tilde{\chi} : \chi \in \widehat{\mathbb{Z}_p^*} \}
\end{equation}

We now can use (\ref{eqn:Fi_orthog}) to get for all $\chi \in \widehat{\mathbb{Z}_p^\ast}$ that 
\begin{eqnarray*}
 \tilde{\chi}(F) & = & \tilde{\chi}(F_1) \\ & = & 
 p \sum_{l=1}^{p-1} \sum_{m=1}^{p-1} a(m) |\varphi(-lm)|^2 \chi(l) \\
 & = & p \sum_{l=1}^{p-1} \sum_{m-1}^{p-1} a(m) \overline{\chi}(m) |\varphi(-lm)|^2 \chi(lm) \\
 & = & p \sum_{l=1}^{p-1} \sum_{m-1}^{p-1} a(m) \overline{\chi}(m) |\varphi(-l)|^2 \chi(l) \\
 & = & p \overline{\chi}(a) c_\varphi(\chi)~.
\end{eqnarray*}
By assumption, $c_\varphi(\chi) \not= 0$, which allows to solve this equation for $\overline{\chi}(a)$, and plug the result into the Fourier inversion on $\mathbb{Z}_p^\ast$, yielding 
\begin{eqnarray*}
 a(k) & = & \frac{1}{p-1} \sum_{\chi \in \widehat{\mathbb{Z}_p^\ast}} \overline{\chi}(a) \chi(k) \\
 & = & \frac{1}{p(p-1)} \sum_{\chi \in \widehat{\mathbb{Z}_p^\ast}}  c_\varphi(\chi)^{-1} \tilde{\chi}(F) \chi(k)~,
\end{eqnarray*}
which is the formula from the first step.

For the second step, we recall that (\ref{eqn:F2}) allows to view $F_2$ as an inverse Fourier transform, which together with (\ref{eqn:Fi_orthog}) gives rise to 
\begin{eqnarray} \nonumber
p^{-1} A'_2 \cdot (C_\varphi')^* & = & \frac{|G|}{d_{\widehat{\pi}_0}} \left( (\mathbf{0}|A'_2) \cdot (0|C_\varphi')^* \right) \\
 & = &  \widehat{\pi}_0(F_2)  \nonumber \\
 & = & \widehat{\pi_0}(F) \label{eqn:F_2} 
\end{eqnarray}

We next show  
\begin{equation}
 \label{eqn:Bphi_Cphi}
 B_{\varphi} = \Omega_0 \cdot C_\varphi' \cdot \Omega_1^T~.
\end{equation}

For this purpose we introduce the change of variables $\omega: \{ 1,\ldots, p-2 \} \to \{ 2,\ldots, p-1 \}$, 
$\omega(n) = 1+n^{-1}$. $\omega$ is a well-defined bijection, with inverse map given by $\omega^{-1}(n) = (n-1)^{-1}$. Thus the map $\Omega_1 : f \mapsto f \circ \omega$ is a well-defined unitary map with inverse $\Omega_1^T$.  
We apply this change of variables to the \textit{row vectors} of $C_\varphi'$, and the formulas  
\[
 (\omega(n) -1)^{-1} = n~,~\omega(n) (\omega(n)-1)^{-1} = n+1  
\] show that  
\begin{eqnarray*}
 \left( C_\varphi' \cdot \Omega_1^T \right) (m,n) & = & \varphi(m (1-\omega(n))^{-1}) \overline{\varphi(m \omega(n) (1-\omega(n))^{-1})} \\
 & = & \varphi(-mn) \overline{\varphi(-m(n+1))} ~.
\end{eqnarray*}
But then left multiplication by $\Omega_0$, effecting a sign flip in $m$, yields 
\[
 \left( \Omega_0 \cdot C_\varphi' \cdot \Omega_1^T \right) (m,n) = B_\varphi(m,n)~,
\]
as desired. Hence (\ref{eqn:Bphi_Cphi}) is established, which implies $C_\varphi' = \Omega_0^T \cdot  B_\varphi \cdot \Omega_1$. In combination with  (\ref{eqn:F_2}) this  yields
\begin{eqnarray*}
 p \cdot \widehat{\pi}_0(F) \cdot \Omega_0^T \cdot (B_\varphi^\dagger)^* \cdot \Omega_1 & = & 
 A_2' \cdot (C_\varphi')^* \cdot  \Omega_0^T \cdot (B_\varphi^\dagger)^* \cdot \Omega_1  \\
 & = & A_2' \cdot \left( \Omega_1^T B_\varphi^* \cdot \Omega_0 \right) \cdot \Omega_0^T \cdot B_\varphi^\dagger)^* \cdot \Omega_1 \\
 & = & A_2' \cdot \Omega_1^T \cdot (B_\varphi^\dagger \cdot B_\varphi)^* \cdot \Omega_1 \\
 & = & A_2'~,
\end{eqnarray*}
by choice of $B_\varphi^\dagger$. 

Hence the recovery of $SA$ from $F$ via steps (1) and (2) is shown, and application of $S^*$ yields the desired result. 

This establishes the inversion algorithm, as well as sufficiency of the conditions (i), (ii). It remains to show necessity of the conditions. First assume that condition (i) is violated, i.e. there exists a character $\chi \in \widehat{\mathbb{Z}_p^*}$ such that 
\[
 \sum_{m =1}^{p-1} |\varphi(-l)|^2 \chi(l) = 0 ~.
\]

Fix $a_1 = \chi \in \mathbb{C}^{\{1,\ldots,p-1 \}}$, and let $A = S^*(a_1|\mathbf{0})$, as well as $F = V_{\varphi \otimes \varphi}^{\varrho_1} A$. Using the above notations and calculations, we then have 
\[
 F(k,l) = F_1(k,l) =  \sum_{m=1}^{p-1} a(m) |\varphi(-lm)|^2 =  \overline{\chi}(l) \sum_{m=1}^{p-1} \chi(lm) |\varphi(-lm)|^2 = 0 
\] for all $(k,l) \in G$. On the other hand, $A \not= 0$, since $S^*$ is unitary. 

Assuming condition (ii) to be violated, we obtain the existence of a nonzero vector $v \in \mathbb{C}^{\{ 1,\ldots,p-2 \}}$ such that $B_\varphi v = 0$. We pick an arbitrary nonzero $w \in \mathbb{C}^{\{1,\ldots,p-1\}}$ and define $A_2 = (w \otimes v) \cdot\Omega_1$, and $A = S^*(\mathbf{0}|A_2)$. Then $A$ is nonzero, but the above computations yield for all $(k,l) \in G$:
\begin{eqnarray*}
 F(k,l) & = & F_2(k,l) \\
  & = & {\rm trace}(A_2 \cdot (C_{\varphi}')^\ast \pi(k,l)^*) \\
  & = & {\rm trace}((w \otimes v) \cdot \Omega_1 \cdot \Omega_1^T  \cdot B_\varphi^* \cdot \Omega_0 \cdot \pi(k,l)^*) \\
  & = & {\rm trace}(\underbrace{(w \otimes B_\varphi v)}_{=0} \cdot \Omega_0 \cdot \pi (k,l)^*) \\
  & = & 0
\end{eqnarray*}
by choice of $v$. This concludes the proof. 
\end{proof}

The remaining open question is whether the conditions for matrix recovery established in the last theorem can actually be fulfilled by a suitably chosen vector $\varphi$. The positive answer is contained in the next theorem. 
Note that Lemma \ref{lem:Zariski} then implies that almost every vector does matrix recovery. 

\begin{theorem} \label{thm:ex_or}
 Let $\varphi \in \mathbb{C}^{\{1,\ldots,p-1\}}$ be given by 
 \[
  \varphi(1) = 1~,~\varphi(2) = 2~, 
 \] in the case $p=3$, and by 
 \[
  \varphi(m) = 1 -\delta_1(m)~,
 \] for $p>3$. Then $\widehat{\pi}_0(G) \varphi$ does matrix retrieval. 
\end{theorem}
\begin{proof}
In the case $p=3$, the matrix $B_\varphi$ has dimensions $\{1,2 \} \times \{ 1 \}$, and it fulfills the rank condition (ii) iff $\varphi(1) \varphi(2) \not= 0$. Condition (i) is fulfilled as soon as $|\varphi(1)| \not= |\varphi(2)|$, and both conditions are fulfilled by $\varphi(1)=1, \varphi(2) = 2$. 

Now assume $p>3$ and let $\varphi(m) = 1 - \delta_1(m)$. Then property $(i)$ is easy to verify: 
For the constant characer $\chi_0$ of $\widehat{\mathbb{Z}_p^{\ast}}$, we find $\langle |\varphi|^2, \chi_0 \rangle = p-2$. Noting that $|\varphi|^2 = \chi_0 - \delta_1$, we find for all remaining characters $\chi$ that 
\[
 \langle |\varphi|^2, \chi \rangle = \langle -\delta_1, \chi \rangle = - \chi(1) \not= 0~. 
\]

Therefore it remains to verify condition (ii), which is clearly equivalent to the equation ${\rm rank}(B_\varphi) = p-2$.  By definition of $\varphi$ we get $B_{\varphi}(m,n) \in \{0, 1 \}$ for all $m,n$, with 
\begin{eqnarray*}
 B_\varphi(m,n) = 0 & \Longleftrightarrow &  \left( mn=1 \right) \lor  \left( m(n+1)=1  \right) \\
  & \Longleftrightarrow & m \in \{ n^{-1},(n+1)^{-1} \}~.
\end{eqnarray*}
Letting $\tilde{B}_\varphi$ denote the matrix obtained by 
\[
 \tilde{B}_\varphi(m,n) = B_\varphi(m^{-1},n)~,
\] we have ${\rm rank}(\tilde{B}_\varphi) = {\rm rank}(B_\varphi)$, and 
\[
\tilde{B}_{\varphi}(m,n) = 0 \Longleftrightarrow m \in \{n,n+1 \}~,
\] with value $1$ for the remaining entries. 

The next step in the proof consists in successive applications of Gaussian elimination steps to the rows of $\tilde{B}_\varphi$: We first subtract the second row from the first, then the third row from the second,$\ldots$, the $(p-1)$th row from the $(p-2)$th. This yields
\[
 \dot{B}_\varphi(m,n) = \left\{ \begin{array}{cc} -1, & m \ge 2, n=m-1 \\ 1, & m \le p-3, n=m+1 \\ 0 & \mbox{ elsewhere } \end{array} \right. 
\]
Removal of the first row (and an index shift) results in the square matrix $M_\varphi \in \mathbb{C}^{\{1,\ldots,p-2\} \times \{1,\ldots,p-2 \}}$, given by 
\begin{equation} \label{eqn:Bdot_vphi}
 M_\varphi (m,n) = \left\{ \begin{array}{cc} -1, & m \le p-2, n=m \\ 1, & m \le p-4,n=m+2 \\ 1, & m=p-2,n \le p-3 \\ 0, & \mbox{otherwise} \end{array} \right.
\end{equation}
i.e. 
\begin{equation}
 M_\varphi = \left( \begin{array}{rrrrrrrr} -1 & 0 & 1 & & &  &  \\
  & -1 & 0 & 1 &  &  &  &  \\
  &  & \ddots & \ddots  & \ddots & & &  \\
  &   &  & \ddots & \ddots  & \ddots & &  \\
  &   &   &       & -1 & 0 & 1  & \\
  &   &   &       &    &  -1 & 0 & 1  \\
  &  &    &        &   &     & -1 & 0 \\
1 & 1 & \ldots & \ldots & \ldots & \ldots & 1 & 0 
 \end{array} \right) ~.
\end{equation}
Since $M_\varphi$ is a submatrix of $\dot{B}_\varphi$, the steps taken so far ensure that the proof is finished once we have shown that $M_\varphi$ is invertible. For this purpose, we make one final Gaussian elimination step to take care of the last row of $M_\varphi$, by letting
\begin{equation} \label{eqn:def_dotM}
 \dot{M}_\varphi(p-2,n) = M_\varphi(p-2,n) + \sum_{j=1}^{(p-3)/2} j M_\varphi(2j-1,n)+j M_\varphi(2j,n)~,
\end{equation} for $n=1,\ldots, p-2$, and $\dot{M}(k,m) = M(k,m)$ for $k< p-2$. 

For $n=p-2$, we then get
\begin{equation} \label{eqn:dotM_diag}
  \dot{M}_\varphi(p-2,n) = 0 + (p-3)/2 \not= 0 ~.
\end{equation}

Now fix $n< p-2$. In case $n\le 2$, there is only a single contribution on the right-hand side, leading to 
\[
\dot{M}(p-2,n) = M(p-2,n)-1 = 0~. 
\]

For $3 \le n \le p-2$, we distinguish between even and odd cases. In the case where $n=2k$, the only nonzero contributions on the right-hand side of (\ref{eqn:def_dotM}) are provided by $j=k$ and $j=k-1$, leading to 
\[
 \dot{M}_\varphi (p-2,n) = 1 + (k-1) \underbrace{M_\varphi(2k-2,2k)}_{=1} + k \underbrace{M_\varphi(2k,2k)}_{=-1} = 0   
\]

In the case $n=2k+1$, the nonzero contributions also come from $j=k$ and $j=k-1$, with corresponding entries $-1$ resp. $1$, resulting again in $\dot{M}_\varphi(p-2,n) = 0$. 

In summary, we find that $\dot{M}_\varphi$ is upper triangular, with nonvanishing diagonal thanks to (\ref{eqn:Bdot_vphi}) and (\ref{eqn:dotM_diag}), and hence $\dot{M}_\varphi$ is invertible. This concludes the proof. 
\end{proof}

The appearance of the group Fourier transform in the matrix recovery algorithm is no coincidence, but rather a byproduct of the approach taken in this paper. The following remark explains these connections in some more detail, and sketches how Corollary \ref{cor:crit_op_ret} can possibly be generalized to other settings.

\begin{remark} \label{rem:Group_fourier}
Recall from Proposition \ref{prop:mr_conjrep} that matrix recovery for group frames requires finding certain cyclic vectors for the conjugation representation
\[
 \varrho(x)(A) = \pi(x) \cdot A \cdot \pi(x)^*~,
\] acting on $\mathbb{C}^{d_\pi \times d_\pi}$. The natural representation-theoretic approach to this problem relies on the decomposition of $\varrho$ into irreducibles, of the type
\[
 \varrho \simeq \bigoplus_{\sigma \in \widehat{G}} m_\sigma \cdot \sigma ~,
\] where $m_\sigma \in \mathbb{N}_0$ denotes the multiplicity with which $\sigma$ occurs in $\varrho$. 

Slightly rewriting this decomposition allows to introduce the group Fourier transform in a rather natural fashion. As a starting point, we make the observation that any cyclic vector $A$ for $\varrho$ yields an injective operator $V_A^\varrho: C^{d_\pi \times d_\pi} \to \mathbb{C}^G$ intertwining $\varrho$ with the left regular representation $\lambda_G$ acting on $\mathbb{C}^G$. A necessary and sufficient condition for the existence of such an intertwining operator is that $m_\sigma$ is less than or equal to the multiplicity of $\sigma$ in $\lambda_G$, and the latter coincides with $d_\sigma$.

Hence, we have obtained $m_\sigma \le d_\sigma$ for all $\sigma \in \widehat{G}$, as a necessary condition for matrix recovery. In this case, one can in effect realize $m_\sigma \cdot \sigma$ by left action of $\sigma$ on a suitable subspace $\mathcal{K}_\sigma \subset \mathbb{C}^{d_\sigma \times d_\sigma}$, acting by left multiplication. This results in an intertwining operator 
\[ S: \mathbb{C}^{d_\pi \times d_\pi} \to \prod_{\sigma \in \widehat{G}} \mathcal{K}_\sigma \]
such that 
\begin{equation}
 S \varrho(x) S^* (A_\sigma)_{\sigma \in \widehat{G}}  = (\sigma(x) \cdot A_\sigma)_{\sigma \in \widehat{G}}~. 
\end{equation} In this realization, we get for arbitrary $A,B \in \mathbb{C}^{d_\pi \times d_\pi}$ with 
$\widehat{A} = (\widehat{A}_\sigma)_{\sigma \in \widehat{G}} =SA$ and $\widehat{B} =  (\widehat{B}_\sigma)_{\sigma \in \widehat{G}} = SB$, that the matrix coefficient $V_A^\varrho B$ can be computed as 
\begin{eqnarray*}
 V_A \varrho B (x) & = & \langle B, \varrho(x) A \rangle \\
  & = & \sum_{\sigma \in \widehat{G}} \langle \widehat{B}_\sigma, \sigma(x) \widehat{A}_\sigma \rangle \\
  & = & \sum_{\sigma \in \widehat{G}} {\rm trace}(\widehat{B}_\sigma \cdot \widehat{A}_\sigma^* \cdot \sigma(x)^*)~, 
\end{eqnarray*}
which shows that the coefficient transform is very closely related to the Fourier inversion formula (\ref{eqn:Fourier_inv}). As a consequence, we can make the following observations, using similar reasoning as in the proof of Theorem \ref{thm:crit_op_ret}: 
\begin{enumerate}
\item[(a)] Invertibility of $V_A^\varrho$ is equivalent to injectivity of $B_\sigma \mapsto B_\sigma \cdot A_\sigma^*$ on each $\mathcal{K}_\sigma$.
\item[(b)] Recovery of $B$ from $F = V^\varrho_A B$ can be carried out as follows:  (i) Compute $\widehat{B}_\sigma \cdot \widehat{A}_\sigma^* = |G| d_\sigma^{-1} \sigma(F)$; (ii) for each $\sigma$, recover $\widehat{B}_\sigma \in \mathcal{K}_\sigma$ from  the result of the previous step (as guaranteed by the condition in (a)); and (iii) invert $S$. 
\item[(c)] The additional challenge posed by matrix retrieval from measurements of the type 
\[
 \langle A \pi(x) \psi, \pi(x) \psi \rangle 
\]
shows itself in applying the criteria and inversion methods from (a) and (b) to specific matrices $A = \psi \otimes \psi$, which depends quadratically on the entries in $\psi$. Note that the explicit formulation of these criteria also hinges on the concrete choice of the intertwining operator $S$. 
\end{enumerate}

We emphasize that these observations apply to any finite group, and they provide a unified representation-theoretic perspective on matrix recovery via group frames. Remark \ref{rem:mr_Heisenberg} below explains how the well-understood case of Heisenberg groups fits in this framework.  Note that this point of view is quite similar to the treatment of admissible vectors (i.e., special continuous group frames arising from unitary actions of locally compact groups, and their relation to the Plancherel formula of such groups), as developed in \cite{MR2130226}. 

The criteria and inversion method in Theorem \ref{thm:crit_op_ret} are closely related to the general program sketched in (a)-(c) above, but they do not constitute a faithful implementation. Proposition \ref{prop:decompose} decomposes the representation space into two orthogonal subspaces, $\mathcal{K}_1$ and $\mathcal{K}_2$. On $\mathcal{K}_1$, $\varrho_2$ is essentially the regular representation of the quotient group $\mathbb{Z}_p^\ast$, whereas the action on $\mathcal{K}_2$ is a multiple of the irreducible representation $\widehat{\pi}_0$. 

Now the decomposition of the regular representation of $\mathbb{Z}_p^\ast$ is effected by its characters, which can also be considered as (lifted) characters of the whole group $G$. Thus the Fourier coefficients $c_\varphi$ occurring in condition (i) from Theorem \ref{thm:crit_op_ret} and in Step (1) of the recovery algorithm could also be expressed in terms of the group Fourier transform of $G$, which would increase the similarity to steps (a) and (b). 

In the case of condition (ii) and recovery step (2) from Theorem \ref{thm:crit_op_ret}, the group Fourier transform is more explicitly employed, in a way that corresponds quite closely to steps (a) and (b) above. The main difference is that we opted to formulate the injectivity condition (ii), which corresponds to (a), in terms of the matrix $B_\varphi$. This matrix has the advantage that its dependence on the vector $\varphi$ is rather transparent. This simplifies verification of condition (ii), as can be seen in the proof of Theorem \ref{thm:ex_or}. The price to pay is the occurrence of the auxiliary matrices $\Omega_0$ and $\Omega_1$ in the inversion step (2).  
\end{remark}

\begin{remark} \label{rem:mr_Heisenberg}
In this remark we study phase retrieval and matrix recovery in the context of finite Gabor systems. 
The following is largely based on observations made in \cite{MR3500231}. The main purpose of this remark, and the slightly novel contribution to the discussion of matrix recovery for the Schr\"odinger representation, is the explicit reference to the program outlined in Remark \ref{rem:Group_fourier}.

 Fix $n \in \mathbb{N}$, $n \ge 2$, and define the associated finite Heisenberg group by
 \[
  \mathbb{H} = \mathbb{Z}_n^3
 \] with group law
 \[
  (k,l,m)\cdot(k',l',m') = (k+k',l+l',m+m'-l'k)~.
 \] The center of this group is given by $Z(\mathbb{H}) = \{ 0 \} \times \{ 0 \} \times \mathbb{Z}_n$, and $\mathbb{H}/Z(\mathbb{H}) \cong \mathbb{Z}_n \times \mathbb{Z}_n$. 
 
 The \textbf{Schr\"odinger representation} of $\mathbb{H}_n$ acts on $\mathbb{C}^n$ via
 \[
  (\pi(k,l,m)f)(y) = e^{2 \pi im/n} e^{2 \pi i ly/n} f(y-k)~,k,l,m,y \in \{0,\ldots,n-1 \}~,
 \] with $y-k$ understood modulo $n$.  The Schr\"odinger representation is irreducible, and the restriction of $\pi$ to the center $Z(\mathbb{H})$ is a character. Since we are interested in matrix retrieval, we will systematically omit the central variable in the following: All scalar actions cancel out in the conjugation action, which means that $Z(\mathbb{H})$ is necessarily in the kernel of $\varrho$. In particular, we will write $\pi(k,l):= \pi(k,l,0)$.
 
 In order to establish criteria for matrix recovery, we make the following two observations:
 \begin{equation} \label{eqn:pi_onb}
  \forall k,l,k',l' ~:~{\rm trace}(\pi(k,l) \pi(k',l')^*) = n \delta_{k,k'} \delta_{l,l'}~,
 \end{equation}
which means that (for dimension reasons) the matrices $(n^{-1/2} \pi(k,l))_{(k,l) \in \mathbb{Z}_n^2}$ constitute an orthonormal basis of $\mathbb{C}^{n \times n}$. The second observation is 
\begin{equation} \label{eqn:pi_ev}
\forall k,l,k',l'~:~ \pi(k,l) \pi(k',l') \pi(k,l)^* = e^{-2 \pi i(kl'+k'l) } \pi(k',l') ~.
\end{equation}
We are interested in the conjugation action $\varrho$, and we therefore immediately rewrite (\ref{eqn:pi_ev}) to obtain
\begin{equation} \label{eqn:rho_ev}
\forall k,l,k',l'~:~ \varrho(k,l) \left( \pi(k',l') \right) = e^{-2 \pi (l'k+lk')/n} \pi(k',l') ~.
\end{equation}
Taking (\ref{eqn:pi_onb}), (\ref{eqn:rho_ev}) together, we obtain the unitary equivalence
\[
 \varrho \simeq \bigoplus_{(k',l') \in \mathbb{Z}_n^2} \chi_{k,l}~,
\] with pairwise distinct homomorphisms $\chi_{k',l'} : \mathbb{H} \to \mathbb{T}$ defined by 
\[
 \chi_{k',l'} (k,l,m) =  e^{-2 \pi (l'k+lk')/n}~.
\]
Hence we have decomposed $\varrho$ into irreducible and pairwise inequivalent representations of $\mathbb{H}$, with 
the intertwining operator $S: \mathbb{C}^{n \times n} \to \mathbb{C}^{n \times n}$ explicitly given by 
\[
 SA = n^{-1/2} (\langle A, \pi(k,l) \rangle)_{(k,l)  \in \mathbb{Z}_n^2}~.
\] 
For rank-one operators of the form $A = \varphi \otimes \varphi$ this specializes to 
\[
 SA (k,l) = n^{-1/2} \langle \varphi \otimes \varphi, \pi(k,l) \rangle = n^{-1/2} \langle \varphi, \pi(k,l) \varphi \rangle~,
\] which is the so-called \textbf{ambiguity function} $A_\varphi$ of $\varphi$. 

In order to see how this function enters into the computation of $F = V_{\varphi \otimes \varphi}^\varrho A$, for an arbitrary matrix $A$, we use the intertwining property of $S$ to compute 
\begin{eqnarray*}
 F(k,l) & = & \sum_{(k',l')} n^{-1} \langle A, \pi(k',l') \rangle \overline{ \langle \varphi, \pi(k',l') \varphi \rangle} \overline{\chi_{k',l'}(k,l)}~. 
\end{eqnarray*}
Since $(\chi_{k',l'})_{k',l'}$ is just the set of characters of the abelian group $\mathbb{Z}_n^2$, this formula for $F$ is recognized as a Fourier transform over $\mathbb{Z}_n^2$. Fourier inversion then yields
 \begin{equation} 
 \langle A, \pi(k',l') \rangle \overline{ \langle \varphi, \pi(k',l') \varphi \rangle} = n^{-1} \sum_{(k,l)} F(k,l) \chi_{k',l'}(k,l) ~. 
\end{equation}
This can be solved for the expansion coefficients $SA = (n^{-1/2} \langle A, \pi(k',l') \rangle)_{k',l'}$ if (and only if) $A_\varphi$ is nowhere vanishing. The remaining step is then to invert $S$.

These observations are summarized as follows:  
\begin{enumerate}
 \item[(a)] $\varphi$ does matrix recovery iff the ambiguity function of $\varphi$ is nowhere vanishing. 
 \item[(b)] Assume that $\varphi$ fulfills the condition from part (a). Then an arbitrary matrix $A \in \mathbb{C}^{n \times n}$ can be recovered from $F: \mathbb{Z}_n^2 \to \mathbb{C}$, 
 \[
  F(k,l) = \langle A, \varrho(k,l) (\varphi \otimes \varphi) \rangle 
 \] 
 by the following steps: 
 \begin{enumerate}
  \item[(i)] Compute $\widehat{F}(k',l') = \sum_{(k,l)} F(k,l) \chi_{k',l'}(k,l)$.
  \item[(ii)] $A$ is given by
  \[
  A = \sum_{(k',l')} n^{-2} \frac{\widehat{F}(k,l)}{\langle \pi(k,l) \varphi, \varphi  \rangle} \pi(k',l')~.
  \]
 \end{enumerate}
\end{enumerate}
We do not claim originality for these observations. The fact that a nonvanishing ambiguity function guarantees phase retrieval for the short-time Fourier transform is well-known.
Sufficiency of criterion (a) for matrix recovery can be found in \cite[Theorem 2.2]{MR4094471}, whereas necessity is implicit in \cite[Theorem 2.3]{MR4094471}. We are not aware of a previous formulation of the recovery procedure from Part (b). However, it is implicit in the proof of  \cite[Theorem 3.16]{MR4094471} (although only for $A = f \otimes f$). 

The parallels between items (a) and (b) in the previous paragraph and the corresponding steps in the program outlined in Remark \ref{rem:Group_fourier} should now become obvious. Possibly the biggest obstacle to the proper understanding of this correspondence comes from the fact that the representation $\varrho$ is trivial on the center of $\mathbb{H}$, i.e. it can be viewed as a representation of the abelian quotient group $\mathbb{H}/\mathbb{Z}(\mathbb{H})$. As a consequence, the decomposition of $\varrho$ only involves characters, and all irreducible subspaces are one-dimensional, which greatly simplifies both the criterion for matrix recovery, and the inversion formula. At the same time, these simplifications somewhat obscure the role of the group Fourier transform of the (nonabelian) Heisenberg group itself in steps (a) and (b). 

In fact, it can be argued with some justification that this role can be safely ignored, since all computations can be carried out over the abelian factor group. However, the slightly more involved point of view developed here allows to appreciate the wider representation-theoretic context, and a more direct comparison to the case of affine groups. 
\end{remark}

It is worthwhile to explicitly compute the generating vector for the unitarily equivalent permutation representation $\pi_0$, using Fourier inversion.
\begin{corollary} \label{cor:mr_pi0}
 Let $p>2$ be a prime, and define $\psi \in \mathbb{C}^{\{0,\ldots,p-1\}}$ by  
 \[
  \psi(k) = e^{2 \pi i k/3}+2e^{4 \pi i k/3}
 \] in the case $p=3$, and 
 \[
  \psi(k) = \delta_0(k) - p^{-1} - p^{-1} e^{2 \pi i k/p}~
 \] for $p \ge 5$. 
Then $\pi_0(G) \psi$ does matrix recovery. 
\end{corollary}

\begin{remark} \label{rem:Pauli_pairs}
 Condition (i) from Theorem \ref{thm:crit_op_ret} is equivalent to requiring that all \textit{diagonal} operators $A \in \mathbb{C}^{\{1,\ldots,p-1 \} \times \{ 1,\ldots, p-1 \}}$ can be recovered from their scalar products with $\varrho_1(k,l) (\varphi \otimes \varphi)$.

This property provides an interesting interpretation of the phase retrieval problem for the permutation  representation $\pi_0$, acting on $\mathcal{H}_0 = \mathbf{1}^\bot$, by the observation that the diagonal operators in Fourier domain are just the convolution operators in time domain. 
 
 Hence condition (i) imposed on $\varphi = \widehat{\psi}$ is equivalent to the fact that all convolution operators $A$ on $\mathcal{H}_0$ can be recovered from their scalar products $\langle A \pi(k,l) \psi, \pi(k,l) \psi \rangle, (k,l) \in G$.

  Now assume to be given vectors $f,g \in \mathcal{H}_0$ with $|V^{\pi_0}_\psi f| = |V^{\pi_0}_\psi g|$.
  Using that $V^{\pi_0}_\psi f (k,l) = f \ast \psi_l$, with $\psi_l(m) = \overline{\psi(-l^{-1} m)}$, and the analogous formula for $V^{\pi_0}_\psi g$, together with the Plancherel formula and convolution theorem for $\mathbb{Z}_p$, allow to compute for all $l \in \mathbb{Z}_p^\ast$ that 
  \begin{eqnarray*}
 \sum_{m=1}^{p-1} |\widehat{f}(m)|^2 |\widehat{\psi}(-lm)|^2 & = &   p^{-1/2} \sum_{k=0}^{p-1} |V_\psi^{\pi_0} f(k,l)|^2 \\
  & = &  p^{-1/2} \sum_{k=0}^{p-1} |V_\psi^{\pi_0} g(k,l)|^2 \\
  & = & \sum_{m=1}^{p-1} |\widehat{g}(m)|^2 |\widehat{\psi}(-lm)|^2~.
  \end{eqnarray*}
Both ends of this equation are convolution products over the multiplicative group. Now condition (i) from Theorem \ref{thm:crit_op_ret} implies that convolution with $|\widehat{\psi}|^2$ is injective; compare the argument following equation (\ref{eqn:Fi_orthog}). This implies
\[
\forall m \in \mathbb{Z}_p ~:~ |\widehat{f}(m)| = |\widehat{g}(m)|~.
\]

With this equality established, the convolution theorem yields
\[
 \left|\left( f \ast \psi_l \right)^\wedge \right| = p^{1/2} |\widehat{f}| \cdot |\widehat{\psi}_l| = \left| \left( g \ast \psi_l \right)^\wedge \right|
\]
This means that if we compare the fixed lines $F_l = V_\psi^{\pi_0} f(\cdot, l), G_l = V_\psi^{\pi_0} g (\cdot, l)$, the assumption $|V_\psi^{\pi_0} f| = |V_\psi^{\pi_0} g|$ is equivalent to the condition
\[
 \forall l=1,\ldots,p-1: |F_l|=|G_l| \mbox{ and } \left| \widehat{F}_l \right| = \left| \widehat{G}_l \right|~.
\]
I.e., we get a family of so-called {\em Pauli pairs}. The study of Pauli pairs and Pauli uniqueness, i.e. the question whether $f$ can be recovered from $|f|$ and $|\widehat{f}|$ is a phase retrieval problem in its own right, with origins in quantum physics; see e.g. \cite{MR1700086,MR4094471}. Note that the whole family $(F_l,G_l)_{l =1,\ldots,p-1}$ of Pauli pairs is generated from a single choice of $f,g$ and $\psi$. 
\end{remark}

\begin{remark} \label{rem:pt_fourier}
Corollary \ref{cor:mr_pi0} allows to derive an intriguing observation regarding phase retrieval for finite Fourier transforms.  Given $l \in \{0,1,\ldots,p-1 \}$, let $\chi_l(m) = e^{2 \pi i l m} \in \mathbb{C}^p$.  Define for each $l \in \{1,\ldots,p-1 \}$ the projection operator $P_l$ onto $\chi_l^\bot \subset \mathcal{H}_0$. In other words, $P_l$ is computed from $f$ by omitting the $l$th frequency from the Fourier series of $f$. 
Then using  
\[ \psi(k) = \delta_0(k) - p^{-1} - p^{-1} e^{2 \pi i k/p}~,~\widehat{\psi} = p^{-1/2}\left(\mathbf{1}-\delta_0 - \delta_1 \right)  \]
and $\psi_l(k) = \overline{\psi(-l^{-1}k)}$ allows to write
\[
 \left( f \ast \psi_l \right)^\wedge = p^{1/2} \widehat{f} \cdot \widehat{\psi}_l = \widehat{f} \cdot \left( \mathbf{1} - \delta_0 - \delta_{l^{-1}} \right)~,
\]
whence 
\[
 V_\psi^{\pi_0} f (\cdot, l) = f \ast \psi_l = P_{l^{-1}}(f)~. 
\]

Thus the fact that $\psi$ does phase retrieval is equivalent to the following fact concerning Fourier phase retrieval:

\textit{For any pair $f,g \in \mathcal{H}_0$: $|P_l f| = |P_l g|$ holds for all $l \in \mathbb{Z}_p^{\ast}$ iff $f \in \mathbb{T} \cdot g$.}

We are not aware of a previous source making this observation, nor of a more elementary proof than via Theorem \ref{thm:crit_op_ret}. It is also unclear to us whether this observation holds for the Fourier transform over general cyclic groups, or only over those of prime order. 
\end{remark}

%
%
%

\section{Phase retrieval for $3$-fold transitive permutation actions}

Recall that as a slight reformulation of the results from the previous section, we have that $S(3)$, acting on $\mathcal{H}_0$, does phase retrieval. We will now use this observation to prove that every $3$-fold transitive subgroup $\Gamma \subset S(n)$ does phase retrieval, when acting on $\mathcal{H}_0$.

We start by noting an auxiliary result which requires additional notation. Given $A \subset \{ 0,\ldots,n-1 \}$, we denote 
\[
 \mathcal{H}_A = \{ g \in \mathcal{H}_0 : {\rm supp}(g) \subset A \}~,
\] and let $Q_A: \mathcal{H}_0  \to \mathcal{H}_A$ denote the orthogonal projection. 
The following is a phase propagation result, in a spirit similar to results in \cite{cheng2020stable}.
\begin{lemma} \label{lem:pr_projections}
 Let $n \ge 3$, and assume that $f,g \in \mathcal{H}_0$ are such that for all three-element subsets $A \subset \{0,\ldots,n-1\}$ there exists $\alpha_A \in \mathbb{T}$ with $Q_A g = \alpha_A Q_A f$. Then there exists $\alpha \in \mathbb{T}$ with $g = \alpha f$. 
\end{lemma}

\begin{proof}
 We proceed by induction, first noting that the case $n=3$ is obvious. For the induction step, assume the desired statement has been established for $\mathbb{C}^{\{0,\ldots,n-1 \}}$, and let $f,g \in \mathcal{H}_0 \subset \mathbb{C}^{\{ 0,\ldots,n \}}$ be given, with  $Q_A g = \alpha_A Q_A f$ for all 3-element subsets $A \subset \{ 0,\ldots,n \}$. 
 
 W.l.o.g. $f$ is nonzero, and since the union of all $\mathcal{H}_A$ corresponding to three-element subsets spans $\mathcal{H}_0$, this implies that $g$ is nonzero, too.
 Furthermore, $f\in \mathcal{H}_0$ implies that the support of $f$ contains at least two elements $l,l'$.
 
 In the case where ${\rm supp}(f) = \{ l,l' \}$, the fact that $Q_A (g) = \alpha_A Q_A(f)$ holds for all 3-element sets of the type $\{ l, l',k \}$ implies $g(k) = 0$ for all $k \not\in \{l,l' \}$, and thus $g = \alpha f$ follows. 
 
 We can therefore assume for the remainder of the proof that ${\rm supp}(f)$ contains at least three elements,
 and we let $B = \{ 0,\ldots,n \} \setminus \{ l \}$ and $B' = \{ 0,\ldots,n \} \setminus \{ l' \}$. 
 
 Now the induction hypothesis, applied with $\mathcal{H}_{B}, \mathcal{H}_{B'}$ and all three-element subsets $A \subset B$ and $A' \subset B'$, respectively, provide us with $\alpha_B,\alpha_{B'}\in \mathbb{T}$ such that 
 \[
  Q_B g = \alpha_B Q_B f ~,~ Q_{B'} g = \alpha_{B'} Q_{B'} f~. 
 \]
 The fact that $\{ l,l' \} \subsetneq {\rm supp}(f)$ implies that $Q_{B} Q_{B'} f = Q_{B' \cap B} f \not= 0$. 
 But then we get on the one hand 
 \[
  Q_{B'} Q_B g = \alpha_B Q_{B'} Q_{B} f =  \alpha_B Q_{B'\cap B} f 
 \] and on the other
 \[
  Q_{B'} Q_B g = Q_B Q_{B'} g = \alpha_{B'} Q_{B' \cap B} f~, 
 \] leading to 
 \[
  \alpha_{B'}  Q_{B' \cap B} f = \alpha_B Q_{B' \cap B} f~,
 \] with nonzero $ Q_{B' \cap B} f$. This finally entails $\alpha_{B'} = \alpha_B=:\alpha$. Now the fact that $\mathcal{H}_{B} \cup \mathcal{H}_{B'}$ spans all of $\mathcal{H}_0$ implies
 \[
  g = \alpha f~. 
 \]
\end{proof}

\begin{theorem}
 Let $\Gamma < S(3)$ denote a $3$-fold transitive permutation group. Pick $\psi_0 \in \mathbb{C}^{\{0,1,2 \}}$ doing phase retrieval for $\mathcal{H}_0$ under the action of $S(3)$, and define $\psi \in \mathbb{C}^{\{0,\ldots,n-1 \}}$ as trivial extension of $\psi_0$. Then $\Pi(\Gamma) \psi$ does phase retrieval on $\mathcal{H}_0$.  
\end{theorem}
\begin{proof}
Fix any three-element subset $A \subset \{0,\ldots, n-1 \}$, and consider the set 
\[
 \{ \Pi(h) \psi ~:~ h \in \Gamma, \Pi(h) \psi \in \mathcal{H}_A \}~.  
\] By a suitable reindexing of $A$ by $\{0,1,2 \}$, this set corresponds to the system $\pi(S(3)) \psi_0$ which was chosen to do phase retrieval on the subspace $\mathbb{C}^{\{ 0,1,2 \}}$ perpendicular to the constant signal. It follows for all $f,g \in \mathcal{H}_0$ satisfying $|V_\psi f| = |V_\psi g|$, and for all three-element subsets $A \subset \{0,1,\ldots,n-1 \}$, that $Q_A g = \alpha_A Q_A f$ holds, with $\alpha_A \in \mathbb{T}$. Hence the assumptions of the previous lemma are fulfilled, and $f = \alpha g$ follows.  
\end{proof}
%
%

\bibliographystyle{abbrv}
\bibliography{pr_affine}

\end{document}